\documentclass{amsart}

\usepackage{cite}
\usepackage{amsmath}
\usepackage{amssymb}
\usepackage{amsthm}   
\usepackage{empheq}   
\usepackage{graphicx}
\usepackage{color}
\usepackage{appendix}
\usepackage{hyperref}

\newtheorem{theorem}{Theorem}[section]
\newtheorem{lemma}[theorem]{Lemma}

\newtheorem{corollary}[theorem]{Corollary}

\theoremstyle{definition}
\newtheorem{definition}[theorem]{Definition}
\newtheorem{example}[theorem]{Example}

\theoremstyle{remark}
\newtheorem{remark}[theorem]{Remark}

\numberwithin{equation}{section}

\def\EE{{\mathcal{E}}}
\def\FF{{\mathcal{F}}}
\def\ss{{\mathtt{s}}}
\def\tt{{\mathtt{t}}}
\def\SS{{\mathbf{S}}}

\begin{document}

\title[On symmetric one-dimensional diffusions]{On symmetric one-dimensional diffusions}

\author{ Liping Li}
\address{RCSDS, HCMS, Academy of Mathematics and Systems Science, Chinese Academy of Sciences, Beijing 100190, China.}
\email{{liliping@amss.ac.cn}}
\thanks{The first named author is partially supported by a joint grant (No. 2015LH0043) of China Postdoctoral Science Foundation and Chinese Academy of Science, China Postdoctoral Science Foundation (No. 2016M590145), {NSFC (No. 11688101) and Key Laboratory of Random Complex Structures and Data Science, Academy of Mathematics and Systems Science, Chinese Academy of  Sciences (No. 2008DP173182)}. The second named author is partially supported by NSFC No. 11271240.}

\author{Jiangang Ying}
\address{School of Mathematical Sciences, Fudan University, Shanghai 200433, China.}
\email{jgying@fudan.edu.cn}





\begin{abstract}
The main purpose of this paper is to explore the structure of local and regular Dirichlet forms associated with  { the symmetric one-dimensional diffusions, which are also called the symmetric linear diffusions}. Let $(\EE,\FF)$ be a regular and local Dirichlet form on $L^2(I,m)$, where $I$ is an interval and $m$ is a fully supported Radon measure on $I$. We shall first present a complete representation for $(\EE,\FF)$, which shows that $(\EE,\FF)$ lives on at most countable disjoint `effective' intervals with `adapted' scale function on each interval, and any point outside these intervals is a trap of the {one-dimensional} diffusion. Furthermore, we shall give a necessary and sufficient condition for $C_c^\infty(I)$ being a special standard core of $(\EE,\FF)$ and identify the closure of $C_c^\infty(I)$ in $(\EE,\FF)$ when $C_c^\infty(I)$ is contained but not necessarily dense in $\FF$ relative to the {$\EE_1^{1/2}$-norm}. This paper is partly motivated by a result of Hamza \cite{H75}, stated in \cite[Theorem~3.1.6]{FOT11} and
provides a different point of view to this theorem. To illustrate our results, many examples are provided.
\end{abstract}

\maketitle

\tableofcontents

\section{Background and introduction}\label{SEC0}

\def\IR{\mathbb{R}}
Though the title of this paper is simple and easy to understand, it is still a little bit misleading because it actually focuses on the structure of regular and local Dirichlet forms on real line. However it is roughly correct in the sense that such Dirichlet forms are in one to one correspondence with the symmetric {one-dimensional diffusions, which are also called the symmetric linear diffusions}. A Dirichlet form is a closed symmetric form with Markovian property on $L^2(E,m)$ space, where $E$ is a nice topological space and $m$ is a Radon measure on $E$. Theory of Dirichlet form, established by a group of French mathematicians lead by Beurling, Choquet and Deny, is closely related to probability theory because of its  Markovian property. We refer the notion and terminology in theory of Dirichlet forms to \cite{CF12, FOT11, MR92}.

What do `regular' and 'local' above mean? On one hand, the `regularity' of a Dirichlet form assures that it is associated with a symmetric Markov process, due to a series of important works by M. Fukushima, M. L. Silverstein in 1970's, and Albeverio, Z. M. Ma, M. R\"ockner in 1990's. Note and be careful that the regularity of a Dirichlet
 form is different from the regularity of a {one-dimensional} diffusion. Intuitively
 the `regularity' of a Dirichlet form means that if there are enough continuous functions in its domain, while the `regularity' of a {one-dimensional} diffusion means roughly that it is irreducible as in \eqref{EQ2PXY}. 
On the other hand, the locality of a Dirichlet form is equivalent to continuity of sample paths of the related Markov process. Hence a regular and local Dirichlet form on (or an interval of) real line corresponds to a {one-dimensional} diffusion, which is, by definition, a strong Markovian process with continuous paths on an interval of real line. One-dimensional Brownian motion is a typical example of {one-dimensional} diffusion. There are rich examples and interesting problems in {one-dimensional} diffusions, which have attracted many famous probabilists such as W. Feller, {K. It\^o and H. Mckean}, and the classical results are now collected in many books, such as \cite{I06, IM74, RW87}. 
One-dimensional diffusion has simple and beautiful structure. It is well known that under the `regularity' condition (see \eqref{EQ2PXY} below), a {one-dimensional} diffusion could be characterized essentially by a function, called scale function, a measure, called speed measure, and the killing measure if there is killing inside (see \cite[Chapter~V~\S7]{RW87}).
It was illustrated in \cite{FHY10} that in this case its speed measure is also the unique symmetric measure. Then regular {one-dimensional} diffusions are under the framework of symmetric Markov processes, which generate naturally Dirichlet forms, or energy forms, through infinitesimal generators of the transition semigroups. 
Therefore regular and local Dirichlet forms on real line may be viewed as symmetric {one-dimensional} diffusions.

As said above, the main purpose of this paper is to explore the structure of regular and local Dirichlet forms on real line. It is motivated by one of our recent paper \cite{LY16} in which we felt that the structure might be characterized in some way. 
We would indicate that this paper is also motivated by an early result of Hamza, see \cite[Theorem~3.1.6]{FOT11} and  \cite{H75}, which gives a precise characterization for a symmetric form defined on $C_c^{\infty}(\IR)$ to be closable. The proof of it is completely analytic. We are interested in and try to find its probabilistic and more natural explanation behind it.

To explore the structure, we are going to find a representation of such Dirichlet forms first.
Representation theorem of any kind is always interesting and important, and lives everywhere in mathematics.
For example, as mentioned above, a {one-dimensional} diffusion may be represented by its scale function and speed measure,
and a L\'evy process may be represented by its L\'evy exponent, or more precisely by a non-negative definite matrix, L\'evy measure and a drift vector. The most important representation theorem in the theory of Dirichlet forms is the famous Beurling-Deny decomposition, which asserts that a regular Dirichlet form $(\EE,\FF)$ on $L^2(E,m)$ may be decomposed into three parts: diffusion, jump, and killing, as follows,
 $$\EE(u,v)=\EE^{(c)}(u,v)+\int_{E\times E\setminus d} (u(x)-u(y))(v(x)-v(y))J(dxdy)+\int_E u(x)v(x)k(dx),$$
 for {$u,v\in\FF\cap C_c(E)$}.
As you see, though jump and killing parts may be characterized clearly by measures $J$ and $k$ respectively, the diffusion part $\EE^{(c)}$ is given by a form with strongly local property and its representation is not completely clear. Therefore it is surely significant to find the representation of strongly local forms in various concrete cases. However such representations have seldom appeared in literatures.
As we have seen, there are rich examples of regular Dirichlet forms in the classical textbook \cite{FOT11}, including Example 1.2.2, corresponding to {one-dimensional} diffusion with natural scale, and Hamza's result, as stated in Theorem 3.1.6. However nothing about representation 
has been mentioned
until two papers \cite {FHY10, F10}, in which the authors gave the representation for a strongly local and irreducible regular Dirichlet form on $\IR$ independently, were published. In this paper, we shall first present a complete representation for a strongly local and regular Dirichlet form on $\IR$, without assuming irreducibility. { For a long time we have paid little attention to non-irreducible case because we believed that a non-irreducible process was simply made of some irreducible parts so that the generalization is trivial. As you will see in Theorem \ref{THM21}, the generalization from irreducible case to non-irreducible is far from routine. Especially, the proof of sufficiency, the regularity of such representation, needs more new ideas and techniques.}

{ The advance in technique is a reason to start this paper, but it is not the most important reason.} The main reason that we are interested in the Dirichlet forms without irreducibility is that they appear naturally in various cases { so that the theory of symmetric {one-dimensional} diffusions would not be complete without including it.} The Hamza's result is an example { in which our new representation theorem may be used to illustrate why his result was formulated this way probabilistically, and perfectly, as shown in Theorem \ref{THM410}.} In the recent paper \cite{LY16}, we find that a regular Dirichlet extension of one-dimensional Brownian motion is associated to a {one-dimensional} diffusion, but not necessarily `regular' or irreducible.
It is proved in \cite[Theorem~3.3]{LY16} that every regular Dirichlet extension of one-dimensional Brownian motion may be essentially decomposed into at most countable disjoint invariant intervals and an exceptional set relative to this extension. 
To formulate the representation of strongly local regular Dirichlet forms on $\IR$, we have to come back to those
classical results for {one-dimensional} diffusions stated in
K. It\^o and H. Jr. McKean's classical book \cite{I06, IM74}, which {shows} us profound understanding about intuitive behaviours of general {one-dimensional} diffusion. Some of their results are essentially used in \cite{LY16} and we shall also review several important lemmas in \S\ref{SEC22}.

After the representation of strongly local regular Dirichlet forms is put forward, we shall treat a series of problems, related to the finer matters in structure, in which we are interested, for example, when the form takes $C_c^\infty(\IR)$ as a core and when the form only contains it. As we answer those problems, the Hazma's result will be more clearly and intuitively shaped.


This paper is organized as follows. In \S\ref{SEC2}, the main result (Theorem~\ref{THM21}) derives a concrete formula to express the regular and strongly local Dirichlet form on an interval. It is shown that the state space can be decomposed into at most countable disjoint intervals, on each of which the {one-dimensional} diffusion is regular and characterized by a scale function. Furthermore, any point outside these intervals is a trap of the diffusion. In other words, the trajectory starting from the point outside the intervals will stay there forever. The special case of Theorem~\ref{THM21} on $\mathbb{R}$ is given in Corollary~\ref{COR212} and a general version with killing inside of Theorem~\ref{THM21} is presented in Theorem~\ref{THM51}.


A classical method to construct a regular Dirichlet space is to start from a closable and Markovian symmetric form with the domain  $C_c^\infty$, see \cite[\S3]{FOT11}. Then the closure of $C_c^\infty$ with respect to this form is a regular Dirichlet form. In \S\ref{SEC3}, we shall explore the similar issue through a different angle, and give a necessary and sufficient condition for space of smooth functions, $C_c^\infty(I)$, to be a special standard core of \eqref{EQ2FUL}, see Theorem~\ref{THM37}. Applying Theorem~\ref{THM37} to Hamza's result (see Theorem~\ref{THM410}), we obtain a more intuitive description for structure of the form \eqref{EQ3DEC} and the associated symmetric {one-dimensional} diffusion, which is generated by the closure of $C_c^\infty({I})$, see Corollary~\ref{COR411}. Moreover, we identify the closure of $C_c^\infty({I})$ in \eqref{EQ2FUL} when $C_c^\infty({I})$ is not necessarily a special standard core of \eqref{EQ2FUL}, see Theorem~\ref{THM414}. The possible killing insides of the associated symmetric {one-dimensional} diffusions are discussed in \S\ref{SEC5}. Combining these results together, we think that the structure of strongly local regular Dirichlet forms on real line is almost completely clear.

\subsection*{Notations}

Let us put some often used notations here for handy reference, though we may restate their definitions when they appear.

For $a<b$, $\langle a, b\rangle$ is an interval where $a$ or $b$ may or may not be contained in $\langle a, b\rangle$.
Notations $dx$ and $|\cdot|$ stand for the Lebesgue measure on $\mathbb{R}$ or an interval throughout the paper if no confusion caused.
The restrictions of a measure $\mu$ and a function $f$ to $I$ are denoted by $\mu|_I$ and $f|_I$ respectively.
The notation `$:=$' is read as `to be defined as'.

For a scale function $\ss$ (i.e. a continuous and strictly increasing function) on some interval $J$, $d\ss$ represents its associated Lebesgue-Stieltjes measure on $J$. Set $\ss(J):=\{\ss(x):x\in J\}$. For two measures $\mu$ and $\nu$, $\mu\ll \nu$ means $\mu$ is absolutely continuous with respect to $\nu$. Given a scale function $\ss$ on $J$ and another function $f$ on $J$, $f\ll \ss$ means $f=g\circ \ss$ for some absolutely continuous function $g$ and $\frac{df}{d\ss}:=g'\circ \ss$. Given an interval $J$,  the classes $C_c(J), C^1_c(J)$ and $C^\infty_c(J)$ denote the spaces of all continuous functions on $J$ with compact support, all continuously differentiable functions with compact support and all infinitely differentiable functions with compact support, respectively.

Let $X=(X_t)_{t\geq 0}$ be a Markov process  associated with a Dirichlet form $(\EE,\FF)$ on $L^2(E,m)$. If $A$ is an invariant set of $X$ (see \cite[\S\ref{SEC21}]{LY16}), then the restriction of $(\EE,\FF)$ to $A$ is denoted by $(\EE^A,\FF^A)$ and the restriction of $X$ to $A$ is denoted by $X^A$. More precisely,
\[
\begin{aligned}
\FF^A:=\{f|_A: f\in \FF\},\quad \EE^A(f|_A,g|_A):=\EE(1_Af,1_Ag),\quad f,g\in \FF,
\end{aligned}
\]
which is the associated Dirichlet form on $L^2(A, m|_A)$ of $X^A$.
 If $U$ is an open subset of $E$, then the part Dirichlet form of $(\EE,\FF)$ on $U$ is denoted by $(\EE_U,\FF_U)$ and the part process of $X$ on $U$ is denoted by $X_U$.

\section{Representations of one-dimensional strongly local Dirichlet forms}\label{SEC2}

This section is devoted to characterize the Dirichlet form associated to a symmetric {one-dimensional} diffusion. We are given an interval
\[
	I=\langle l, r\rangle,
\]
where $l$ or $r$ may or may not be in $I$, and a fully supported Radon measure $m$ on $I$. Let $(\EE,\FF)$ be a regular and strongly local Dirichlet form on $L^2(I,m)$. Clearly, its associated Hunt process $X$ is a symmetric diffusion process on $I$ with no killing inside  {due to the strong locality}. The assumption of no killing inside has no essential influence on our results as we will explain later. A trap, by its name, is a point that $X$, starting from it, never leaves. Actually as illustrated later, the interior of $I$ has no trap that $X$ could reach from other points, and however the boundary could act like a trap that $X$ could reach from other points, if $I$ is open. See the examples in \S\ref{SEC231}. Hence we need to pay extra attention to the boundaries of $I$. When $X$ is irreducible (or `regular' in the terminology of \cite[(45.2)]{RW87}) in the sense that
\begin{equation}\label{EQ2PXY}
	\mathbf{P}_x(\sigma_y<\infty)>0,\quad \forall x,y \in I,
\end{equation}
where $\sigma_y$ is the hitting time of $\{y\}$ relative to $X$, its Dirichlet form $(\EE,\FF)$ is characterized by its scale function as in \cite[\S2.2.3]{CF12}, \cite{FHY10} and \cite{F14}. However, the irreducibility is not always right, and we raised many reducible examples such as the regular Dirichlet extensions of one-dimensional Brownian motion in a recent study \cite{LY16}. A natural and interesting question is how to describe this Dirichlet form without the irreducibility. This is what we shall do in this section.

\subsection{Main result}\label{SEC21}

Before presenting the main result, we need to prepare some terminologies and notations. Let $J:=\langle a, b\rangle\subset I$ be an interval, where $a$ or $b$ may or may not be in $J$. Take a fixed point $e$ in $J$ as follows
\[
e:=\left\lbrace \begin{aligned}
 &\frac{a+b}{2},\quad |a|+|b|<\infty, \\
 &a+1,\quad a>-\infty, b=\infty, \\
 &b-1,\quad a=-\infty, b<\infty, \\
 &0,\quad a=-\infty, b=\infty. 	
 \end{aligned}
 \right.
\]
Denote all scale functions on $J$ by $\SS(J)$ 
\[
	\SS(J):=\left\{\ss: J\rightarrow \mathbb{R}: \ss\text{ is continuous and strictly increasing}, \; \ss(e)=0 \right\},
\]
where we set $\ss(e)=0$ to assure the uniqueness.
When $a$ or $b$ is not in $J$, we define $\ss(a)$ or $\ss(b)$ for $\ss\in \SS(J)$ to be the limit:
\[
\ss(a):=\lim_{x\downarrow a}\ss(x)\geq -\infty,\quad \ss(b):=\lim_{x\uparrow b}\ss(x)\leq \infty.
\]
We say that a scale function $\ss\in \SS(J)$ is adapted (to $J$) if 
\begin{itemize}
\item[\textbf{(A)}] When $a>l$ or $a=l\in I$, $\ss(a)=-\infty$ if and only if $a\notin J$;
\item[\textbf{(B)}] When $b<r$ or $b=r\in I$, $\ss(b)=\infty$ if and only if $b\notin J$.
\end{itemize}
What we need is the following subset of $\SS(J)$:
\begin{equation}\label{EQ2SJS}
\SS_\infty(J):=\left\{\ss\in \SS(J): \ss \text{ is adapted}\right\}.
\end{equation}

For the Radon measure $m$ on $I$, let $m_J:=m|_J$ be its restriction to $J$. We write $m_J(\{b-\})=\infty$ (resp. $<\infty$) to stand for that for some small constant $\epsilon>0$, $m_J((b-\epsilon, b))=\infty$ (resp. $<\infty$). The notation $m_J(\{a+\})=\infty$ or $<\infty$ has the similar meaning. Note that $m_J(\{b-\})=\infty$ (resp. $(m_J(\{a+\})=\infty$) might happen only when $b=r\notin I$ (resp. $a=l\notin I$) since $m$ is Radon on $I$.

Next we introduce a Dirichlet form on $L^2(J,m_J)$ induced by a scale function $\ss\in \SS_\infty(J)$. The conditions \textbf{(A)} and \textbf{(B)} has no restriction on $\ss$ for the cases $a=l\notin I$ and $b=r\notin I$. Thus the following situations may happen:
\begin{itemize}
\item[(L0)] $a=l\notin I$, $\ss(a)>-\infty$ and $m_J(a+)<\infty$;
\item[(R0)] $b=r\notin I$, $\ss(b)<\infty$ and $m_J(b-)<\infty$,
\end{itemize}
and in either case, we need to set up the Dirichlet boundary condition.
Define a quadratic form $(\EE^{(\ss)}, \FF^{(\ss)})$ on $L^2(J,m_J)$ as follows:
\begin{equation}\label{EQ2FSU}
\begin{aligned}
&\FF^{(\ss)}:=\bigg\{u\in L^2(J,m_J): u\ll \ss, \; \frac{du}{d\ss}\in L^2(J,d\ss); \\
 &\qquad \qquad\qquad  u(a)=0\text{ (resp. }u(b)=0\text{) whenever (L0) (resp. (R0)) holds} \bigg\},  \\
& \EE^{(\ss)}(u,v):=\frac{1}{2}\int_J \frac{du}{d\ss}\frac{dv}{d\ss}d\ss,\quad u,v\in \FF^{(\ss)}.
\end{aligned}
\end{equation}
Note that if $\ss(a)>-\infty$ (resp. $\ss(b)<\infty$), then any function $u$ satisfying $u\ll \ss$ and $du/d\ss \in L^2(J,d\ss)$ admits the limit $u(a):=\lim_{x\downarrow a}u(x)$ (resp. $u(b):=\lim_{x\uparrow b}u(x)$) (Cf. \cite[\S2.2.3]{CF12}). From \cite[Theorem~2.2.11]{CF12}, it follows that $(\EE^{(\ss)}, \FF^{(\ss)})$ is a regular and strongly local Dirichlet form on $L^2(J,m_J)$ and its associated diffusion process $X^{(\ss)}$ is given by the scale function $\ss$, the speed measure $m_J$ and no killing inside. Moreover, when (L0) or (R0) happens, $a$ or $b$ is an absorbing boundary for $X^{(\ss)}$ and $X^{(\ss)}$ is also called a minimal diffusion on $J$ in \cite[Example~3.5.7]{CF12}. For more discussions about the boundary behaviours of $X^{(\ss)}$, we refer to \cite[\S2.2.3]{CF12}, \cite{F14} and \cite{LY16}.

We are now in a position to state the main result of this section.

\begin{theorem}\label{THM21}
Let $I=\langle l, r\rangle$ be an interval as above and $m$ a fully supported Radon measure on $I$. Then $(\EE, \FF)$ is a regular and strongly local Dirichlet form on $L^2(I,m)$ if and only if there exist a set of at most countable disjoint intervals $\{I_n=\langle a_n,b_n\rangle: I_n\subset I, n\geq 1\}$ and a scale function $\ss_n\in \SS_\infty(I_n)$ for each $n\geq 1$ such that
\begin{equation}\label{EQ2FUL}
\begin{aligned}
	&\FF=\left\{u\in L^2(I,m): u|_{I_n}\in \FF^{(\ss_n)}, \sum_{n\geq 1}\EE^{(\ss_n)}(u|_{I_n}, u|_{I_n})<\infty  \right\},  \\
	&\EE(u,v)=\sum_{n\geq 1}\EE^{(\ss_n)}(u|_{I_n}, v|_{I_n}),\quad u,v \in \FF,
\end{aligned}
\end{equation}
where for each $n\geq 1$, $(\EE^{(\ss_n)}, \FF^{(\ss_n)})$ is given by \eqref{EQ2FSU} with the scale function $\ss_n$ on $I_n$.  Moreover, the intervals $\{I_n: n\geq 1\}$ and scale functions $\{\ss_n:n\geq 1\}$ are uniquely determined, if the difference of order is ignored.
\end{theorem}

\begin{remark}
We would like to give several remarks before the proof of this theorem.
\begin{itemize}
\item[(1)] A special case of Theorem~\ref{THM21} is considered in \cite[Theorem~3.3]{LY16}, in which $(\EE,\FF)$ is a so-called regular Dirichlet extension of one-dimensional Brownian motion. It turns out that in \cite{LY16} $(\bigcup_{n\geq 1}I_n)^c$ has to be of zero Lebesgue measure ($m$ is the Lebesgue measure in this case) and an additional condition on the scale function is imposed. For a general symmetric {one-dimensional} diffusion, these restrictions are not necessarily required.
\item[(2)] Though the intervals are disjoint, they may have common endpoints. For example, $I_m=(a_m,b_m]$ and $I_n=(a_n,b_n)$ with $b_m=a_n$.
\item[(3)] The set of intervals $\{I_n: n\geq 1\}$ is allowed to be empty. Particularly, $\FF=L^2(I,m)$ and $\EE(u,v)=0$ for any $u,v\in \FF$. Then $X$ is a trivial diffusion process and every point of $I$ is a trap of $X$.
\end{itemize}
\end{remark}

\subsection{Proof of Theorem~\ref{THM21}}\label{SEC22}
Recall that two $m$-symmetric Hunt processes are \emph{equivalent} if they possess a common properly exceptional set outside which their transition functions coincide (Cf. \cite[Theorem~4.2.8]{FOT11}). Any regular and strongly local Dirichlet form on $L^2(I,m)$ is associated with a class of equivalent $m$-symmetric diffusions, in which there is a representative $X$ satisfying
\begin{equation}\label{EQ2PXX}
	\mathbf{P}_x(X_{\zeta-}\in I, \zeta<\infty)=0
\end{equation}
for every $x\in I$, where $\zeta$ is the lifetime of $X$. This property is usually called `no killing inside'.

We shall use the same notations as in \cite[\S3.2]{LY16} to classify the state space $I$ of the {one-dimensional} diffusion $X$ as follows (Cf. also \cite{IM74}): 
\begin{enumerate}
\item $\Lambda_2$: regular points, which could reach both sides,
\item $\Lambda_r\cup \Lambda_l$: singular points, which could reach at most one side,
\item $\Lambda_l$/$\Lambda_r$: left/right singular points, which could not reach right/left side,
\item $\Lambda_{pl}$/$\Lambda_{pr}$: left/right shunt points, which could reach only left/right side,
\item $\Lambda_t$: traps, the points which could not reach either side.
\end{enumerate}
Keep in mind the properties of these classes are presented in \cite[Lemma~3.7]{LY16}. Especially, we can deduce the following lemma under the symmetry as in \cite[Lemmas~3.8 and 3.9]{LY16}.

\begin{lemma}\label{LM22}
\begin{itemize}
\item[(1)] Fix a right (resp. left) singular point $b\in \Lambda_{r}$ (resp. $\Lambda_{l}$). Under the symmetry, for any $a<b$ (resp. $a>b$), it holds that
\[
	\mathbf{P}_a(\sigma_b<\infty)=0,
\]
where $\sigma_b$ is hitting time of $\{b\}$ relative to the diffusion process.
Particularly, under the symmetry if $b\in \Lambda_t$ then for any $a\neq b$,
\[
	\mathbf{P}_a(\sigma_b<\infty)=0.
\]
\item[(2)] Under the symmetry, if $(a,b)\subset \Lambda_r$ (resp. $\Lambda_l$), then $(a,b)\subset \Lambda_t$.
\end{itemize}

\end{lemma}

\begin{remark}
Although the discussions in \cite{LY16} are made on the one-dimensional Euclidean space $\mathbb{R}$, all the results in \cite[\S3.3.1-3.3.2]{LY16} still hold for the diffusion on $I$ with no killing inside. For the classification of the points, every point $x$ in the interior of $I$ (i.e. $x\in (l,r)$) can be classified as \cite[Definition~3.6]{LY16}. For the possible $x=r\in I$, formally we say $x$ is a left singular point, i.e. $x\in \Lambda_l$, and we have
\[
	e^-:=\mathbf{P}_x(\sigma_{x-}=0)=0\text{ or }1,
\]
where $\sigma_{x-}:=\inf\{t>0:X_t<x\}$, by the Blumenthal 0-1 law. We say $x$ is a left shunt point, i.e. $x\in \Lambda_{pl}$ if $e^-=1$ and a trap, i.e. $x\in \Lambda_t$ if $e^-=0$. Particularly, if $x=r$ is a trap, then
\[
	\mathbf{P}_x(X_t=x,\forall t\geq 0)=1	
\]
due to `no killing inside'.
Similarly, we can classify the possible $x=l\in I$ as a right shunt point or a trap.
\end{remark}

\subsubsection{Proof of necessity}\label{SEC221}

Let $(\EE,\FF)$ be a regular and strongly local Dirichlet form on $L^2(I,m)$ and $X$ a representative of the associated $m$-symmetric diffusions satisfying \eqref{EQ2PXX} for every $x\in I$.

First note that the regular set $\Lambda_2$ is open (Cf. \cite[Lemma~3.7~(6)]{LY16}). Thus we may write $\Lambda_2$ as a union of disjoint open intervals:
\begin{equation}\label{EQ2LNA}
	\Lambda_2=\bigcup_{n\geq 1} (a_n,b_n).
\end{equation}
Let $$F:=\Lambda_2^c=\Lambda_l\cup \Lambda_r,$$ which is closed. Further let ${\mathring{F}}$ be the interior of $F$ and $\partial F: =F\setminus {\mathring{F}}$. Clearly, $\{a_n,b_n: n\geq 1\}\cap I\subset \partial F$, $\partial F$ is nowhere dense and ${\mathring{F}}$ is open. The proof of necessity will be completed
in four steps.
\medskip

\noindent\textbf{Step 1:} Every point in $\mathring{F}$ is a trap.

\begin{lemma}\label{LM24}
${\mathring{F}}\subset \Lambda_t$.
\end{lemma}
\begin{proof}
Take $x\in (c,d)\subset {\mathring{F}}$. Suppose that $x\in \Lambda_{pr}$. Then there exists a small constant $0<\epsilon<|d-x|$ such that $(x,x+\epsilon)\subset \Lambda_r$. In fact, if this assertion is not right, then there is a decreasing sequence $\{x_k: k\geq 1\}\subset \Lambda_{pl}$ in $(x,d)$ satisfying $x_k\downarrow x$ as $k\rightarrow \infty$. By \cite[Lemma~3.7~(5)]{LY16}, we have $x\in \Lambda_l$, which contradicts the fact $x\in \Lambda_{pr}$.

Since $(x,x+\epsilon)\subset \Lambda_r$, it follows from Lemma~\ref{LM22}~(2) that $(x,x+\epsilon)\subset \Lambda_t$. But $x\in \Lambda_{pr}$ implies that for some $y>x$ (Cf. \cite[Lemma~3.7~(4)]{LY16})
\[
	\mathbf{P}_x(\sigma_y<\infty)>0.
\]
Clearly, $y$ can be taken in $(x,x+\epsilon)$. This contradicts Lemma~\ref{LM22}~(1). Therefore, we have $x\notin \Lambda_{pr}$. Similarly we can also deduce that $x\notin \Lambda_{pl}$. Finally we conclude $x\in \Lambda_t$. That completes the proof.
\end{proof}

\noindent\textbf{Step 2:} We treat the points in $\partial F$.

Since ${\mathring{F}}$ is open, we also write it as a union of disjoint open intervals:
\begin{equation}\label{EQ2FMC}
\mathring{F}=\bigcup_{m\geq 1} (c_m,d_m).
\end{equation}
Then $\partial F$ may be divided into three families: $\{a_n,b_n: n\geq 1\}\cap I$, $\{c_m,d_m:m\geq 1\}\cap \partial F$ and others. We denote them by $\partial_1 F$, $\partial_2 F$ and $\partial_3 F$ respectively.

Consider a point $b_n\in \partial_1 F$ for some $n$. There are the following possible situations:
\begin{itemize}
\item[(1)] $b_n \in \Lambda_t$. Notice the first assertion in Lemma~\ref{LM22}.
\item[(2)] $b_n \in \Lambda_{pl}$. This implies that $\mathbf{P}_{b_n}(\sigma_y<\infty)>0$ for any $y\in (a_n,b_n)$ (Cf. Lemma~3.7~(4) of \cite{LY16}), but $\mathbf{P}_{b_n}(\sigma_z<\infty)=\mathbf{P}_z(\sigma_{b_n}<\infty)=0$ for any $z>b_n$ (Cf. \cite[Lemma~3.7~(2)]{LY16} and Lemma~\ref{LM22}~(1)). Hence if $\mathbf{P}_y(\sigma_{b_n}<\infty)>0$ for some (equivalently, all) $y\in (a_n,b_n)$, then we have
\begin{equation}\label{EQ2PYB}
	\mathbf{P}_y(\sigma_{b_n}<\infty)\cdot \mathbf{P}_{b_n}(\sigma_y<\infty)>0, \quad \forall y\in (a_n,b_n).
\end{equation}
Otherwise, $m(\{b_n\})=0$ and $\{b_n\}$ is an $m$-polar set relative to $X$ by Remark~\ref{RM25} below.
\item[(3)] $b_n\in \Lambda_{pr}$. Clearly, for any $y\in (a_n,b_n)$, $\mathbf{P}_y(\sigma_{b_n}<\infty)=\mathbf{P}_{b_n}(\sigma_y<\infty)=0$. Moreover, by Remark~\ref{RM26} below, we know that $b_n=a_j$ for some $j$. In addition, similar to Remark~\ref{RM25}, either it holds
\begin{equation}\label{EQ2PYA}
	\mathbf{P}_y(\sigma_{a_j}<\infty)\cdot \mathbf{P}_{a_j}(\sigma_y<\infty)>0, \quad \forall y\in (a_j,b_j)
\end{equation}
or $\{b_n\}$ is an $m$-polar set relative to $X$.
\end{itemize}
The point $a_n\in \partial F_1$ can be treated similarly as above.

\begin{remark}\label{RM25}
Let us prove the following fact: if $b_n\in \Lambda_{pl}$ and $\mathbf{P}_y(\sigma_{b_n}<\infty)=0$ for all $y\in (a_n,b_n)$, then $m(\{b_n\})=0$ and thus $\{b_n\}$ is $m$-polar relative to $X$. Suppose that $m(\{b_n\})>0$. Recall that the probability transition semigroup $(P_t)_{t\geq 1}$ satisfies the following symmetric property:
\begin{equation}\label{EQ2PTF}
	(P_tf,g)_m=(f,P_tg)_m,\quad f,g\in \mathcal{B}_+(I), t\geq 0,
\end{equation}
where $(f,g)_m$ stands for the inner product of $L^2(I,m)$ and $\mathcal{B}_+(I)$ is the set of all non-negative Borel measurable functions on $I$. Take $f(x):=1_{\{b_n\}}(x)$ and
\[
	g(x):=\left\lbrace \begin{aligned}
	&-x+b_n,\quad x\leq b_n,\\
	&0,\quad x>b_n.
	\end{aligned} \right.
\]
Since for any $x\neq b_n$, $$P_tf(x)=\mathbf{P}_x(X_t=b_n)=0$$ and $g(b_n)=0$, the left side of \eqref{EQ2PTF} equals $0$. It follows from $m(\{b_n\})>0$ that $P_tg(b_n)=0$, which implies $\mathbf{P}_{b_n}(X_t<b_n)=0$ for any $t\geq 0$. It follows that
\[
	\mathbf{P}_{b_n}\left(\bigcup_{t\in \mathbb{Q}_+}\left\{X_t<b_n \right\} \right)=0
\]
and
\[
	1=\mathbf{P}_{b_n}\left(\bigcap_{t\in \mathbb{Q}_+}\left\{X_t\geq b_n \right\} \right)=\mathbf{P}_{b_n}\left(\bigcap_{t\geq 0}\left\{X_t\geq b_n \right\} \right).
\]
This indicates $b_n\in \Lambda_r$, which contradicts the fact $b_n\in \Lambda_{pl}$.
\end{remark}

\begin{remark}\label{RM26}
For the case $b_n\in \Lambda_{pr}$, we show that $b_n=a_j$ for some $j$. If not, then either $b_n=c_k$ (see \eqref{EQ2FMC}) for some $k$ or  a subsequence $\{(a_{j'},b_{j'})\}$ of \eqref{EQ2LNA} decreases to $b_n$. Both of them are impossible. For the former one, it follows from Lemma~\ref{LM24} that $(c_k,d_k)\subset \Lambda_t$. Then $\mathbf{P}_{b_n}(\sigma_y<\infty)=0$ for any $y>b_n$ by Lemma~\ref{LM22}~(1). This contradicts that $b_n\in \Lambda_{pr}$. The latter one also contradicts that $b_n\in \Lambda_{pr}$ since it implies that $\mathbf{P}_{b_n}(\sigma_y<\infty)>0$ for some $y>b_n$.
\end{remark}

Next, consider $d_m\in \partial_2 F$ (the case $c_m\in \partial_2 F$ is similar). Clearly $d_m\in \Lambda_r$ since $(c_m,d_m)\subset \Lambda_t$. Moreover, either $d_m=a_j$ for some $j$ or a subsequence $\{(a_{j'},b_{j'})\}$ of \eqref{EQ2LNA} decreases to $d_m$. Similar to the third case of $b_n$ above, we can deduce that either of the followings may happen:
\begin{itemize}
\item[(1)] $d_m\in \Lambda_t$.
\item[(2)] $d_m\in \Lambda_{pr}$. Then it must hold that $d_m=a_j$ for some $j$. In addition, either it holds
\[
	\mathbf{P}_y(\sigma_{a_j}<\infty)\cdot \mathbf{P}_{a_j}(\sigma_y<\infty)>0, \quad \forall y\in (a_j,b_j)
\]
or $\{d_m\}$ is $m$-polar relative to $X$.
\end{itemize}

Finally, we assert that $\partial_3 F\subset \Lambda_t$. In fact, for any $x\in \partial_3  F$, there exist two subsequences of $\{(a_n,b_n): n\geq 1\} \cup \{(c_m,d_m):m\geq 1\}$ such that one decreases to $x$ and the other increases to $x$. This excludes the possibility of $x\in \Lambda_{pr} \cup \Lambda_{pl}$. Therefore $\partial F_3\subset \Lambda_t$.

In a word, except for an $m$-polar set, the point in $F$ either belongs to $\Lambda_t$, or is the endpoint of some $(a_n,b_n)$ with the property \eqref{EQ2PYB} or \eqref{EQ2PYA}. Furthermore, the excluded $m$-polar set is contained in $\{a_n,b_n:n\geq 1\}\cap I$.
\medskip

\noindent\textbf{Step 3: } Irreducible intervals and the diffusions on them.

Add the endpoints of every $(a_n,b_n)$ with the property \eqref{EQ2PYB} or \eqref{EQ2PYA} into $(a_n,b_n)$ and we obtain a new interval $I_n:=\langle a_n,b_n\rangle$. First note that $I_n\cap I_m=\emptyset$ for $n\neq m$. In fact, $I_n\cap I_m\neq \emptyset$ only occurs when $a_n=b_m$ or $a_m=b_n$. This concludes that either $a_n=b_m\in \Lambda_2$ or $a_m=b_n\in \Lambda_2$, which contradicts that \eqref{EQ2LNA}. We shall characterize the restriction of $X$ to $I_n$.

Fix $n$ and let us consider the interval $I_n$. From the second step, we know that
\[
\begin{aligned}
	&\mathbf{P}_x(\sigma_y<\infty)=0,\quad \forall x\in I_n, y\in I\setminus I_n,	\\
	&\mathbf{P}_y(\sigma_x<\infty)=0,\quad \forall x\in I_n, \text{ q.e. } y\in I\setminus I_n,
\end{aligned}	
\]
where the excluded possible $m$-polar set is contained in $\{a_n,b_n\}$. Particularly, $I_n$ is an invariant set of $X$ (Cf. \cite[\S3.1]{LY16}). Thus the restriction $X^{I_n}$ of $X$ to $I_n$ is still a diffusion process on $I_n$ and symmetric with respect to $m_n:=m|_{I_n}$ with no killing inside.  Furthermore, $X^{I_n}$ is irreducible in the sense that
\[
	\mathbf{P}_x(\sigma_y<\infty)>0,\quad \forall x,y \in I_n.
\]
It follows that $X^{I_n}$ is characterized uniquely by a scale function $\ss_n\in \SS(I_n)$ since the speed measure $m_n$ is fixed (Cf. \cite[\S3]{FHY10}). We assert that $\ss_n\in \SS_\infty(I_n)$. For the condition \textbf{(A)}, if $a_n>l$, then $a_n\notin I_n$ implies  $\mathbf{P}_x(\sigma_{a_n}<\infty)=0$ for any $x\in I_n$, which is equivalent to (Cf. \cite[(3.5.13)]{CF12})
\[
	\int_{a_n}^c m_n((x,c))d\ss_n(x)=\infty
\]
for $c\in (a_n,b_n)$. Note that $m_n((a_n,c))<\infty$. Thus $\ss_n(a_n)=-\infty$. Similarly we can also prove that $a_n\in I_n$ implies $\ss_n(a_n)>-\infty$. For the case $a_n=l\in I$, we need only to note that $$0<m_n((a_n,c))\leq m([l, c])<\infty.$$ The condition \textbf{(B)} can be verified similarly.
Therefore, it follows from the discussions above and \cite[Theorem~3.1]{FHY10} (see also \cite[\S2.2.3 and Theorem~3.5.8]{CF12} that the associated Dirichlet form of $X^{I_n}$ is given by \eqref{EQ2FSU} with $\ss_n, I_n$ instead of $\ss, J$ respectively.

\medskip
\noindent
\textbf{Step 4: } The expression \eqref{EQ2FUL} of $(\EE,\FF)$ can be deduced by \cite[Corollary~3.11]{LY16} and the following fact: $P_tf(x)=f(x)$ for $x\in \Lambda_t$. That completes the proof of necessity.

\subsubsection{Proof of sufficiency}

Let $(\EE,\FF)$ be given by \eqref{EQ2FUL}. Similar to \cite[\S3.2.5]{LY16}, we can deduce that $(\EE,\FF)$ is a Dirichlet form in the wide sense on $L^2(I,m)$ by \cite[Theorem~3.10]{LY16}. It suffices to prove the regularity.
We shall do it in two steps.
\medskip

\noindent\textbf{First step:} To prove that $\FF\cap C_c(I)$ is dense in $C_c(I)$ relative to the uniform norm.

For each $n$, take another scale function $\tilde{\ss}_n\in \SS(I_n)$ such that
\begin{align*}
	&d\tilde{\ss}_n\ll d\ss_n, \\
&\frac{d\tilde{\ss}_n}{d\ss_n}=0\text{ or }1, \; d\ss_n\text{-a.e.},\\
&|\tilde{\ss}_n(b_n)-\tilde{\ss}_n(a_n)|\leq |b_n-a_n|.
\end{align*}
For the existence of $\tilde{\ss}_n$, we refer to \cite[Remark~3.1]{LY16}. Define a measure $\lambda$ on $I$ as follows:
\[
	\lambda:=\sum_{n\geq 1}d\tilde{\ss}_n+dx|_{\left(\cup_{n\geq 1} I_n\right)^c}.
\]
One may easily check that $\lambda$ is a Radon measure on $I$ and we denote its associated increasing function by $\ss$, i.e. $\lambda=d\ss$. Without loss of generality, we impose $\ss(e)=0$ where $e$ is a fixed point in $I$. We assert that $\ss\in \SS(I)$.  In fact, since $\lambda(\{x\})=0$ for any $x\in I$, $\ss$ is continuous. On the other hand, $\lambda((c,d))>0$ for any $c,d\in I$ with $c<d$ by $\tilde{\ss}_n\in \SS(I_n)$. This implies that $\ss$ is strictly increasing.

Consider now the following class of continuous functions:
\[
	\mathcal{C}:=\{\varphi\circ \ss: \varphi\in C_c^\infty(J)\},
\]
where $J:=\ss(I)=\langle \ss(a),\ss(b)\rangle$ and $\ss(a)\in J$ (resp. $\ss(b)\in J$) if and only if $a\in I$ (resp. $b\in I$). Since $\ss$ is strictly increasing, it follows from Stone-Weierstrass theorem that $\mathcal{C}$ is dense in $C_c(I)$ relative to the uniform norm. We need only to prove $\mathcal{C}\subset \FF\cap C_c(I)$. Take a function $u=\varphi\circ \ss$. Clearly, $u\in L^2(I,m)$. For each $n$, $u|_{I_n}=\varphi\circ \ss|_{I_n}\ll\ss_n$ since $d\ss|_{I_n}=d\tilde{\ss}_n\ll d\ss_n$. Moreover,
\[
\begin{aligned}
	\sum_{n\geq 1} \int_{I_n}\left(\frac{du|_{I_n}}{d\ss_n}\right)^2d\ss_n&=\sum_{n\geq 1} \int_{I_n}\left(\frac{du|_{I_n}}{d\tilde{\ss}_n}\right)^2d\tilde{\ss}_n \\
	&=\sum_{n\geq 1} \int_{I_n}\left(\varphi'\circ\ss\right)^2d\tilde{\ss}_n
\\&\leq \int_I (\varphi'\circ \ss)^2d\ss\\&<\infty.
\end{aligned}
\]
For the possible situations (L0) and (R0), clearly we have $u(a)=0$ and $u(b)=0$ respectively. Hence we conclude that $u\in \FF$.
\medskip

\noindent\textbf{Second step:} To prove that $\FF\cap C_c(I)$ is dense in $\FF$ relative to the norm $\|\cdot\|_{\EE_1}$.

Take $u\in \FF$ and set $$u_1:= u\cdot 1_{\cup_{n\geq 1}I_n},\ u_2:=u-u_1.$$ One may see $u_1,u_2\in \FF$ and $\EE(u_2,u_2)=0$. Fix a small constant $\epsilon>0$. What we shall do is to find a function $u_\epsilon \in \FF\cap C_c(I)$ such that $\EE_1(u_\epsilon-u,u_\epsilon-u)<C\cdot\epsilon$ for some constant $C$.

We first treat $u_1$. Since $(\EE^{(\ss_n)},\FF^{(\ss_n)})$ is regular on $L^2(I_n,m_n)$ and $\EE_1(u_1,u_1)<\infty$, we can take an integer $N$ large enough and for each $1\leq n\leq N$ a function $g_n\in \FF^{(\ss_n)}\cap C_c(I_n)$ such that
\begin{align*}
	&\sum_{n> N}\EE_1^{(\ss_n)}(u|_{I_n}, u|_{I_n}) <\epsilon,\quad\\ &\EE^{(\ss_n)}_1(u|_{I_n}-g_n,u|_{I_n}-g_n)<\epsilon/N,\; 1\leq n\leq N.
\end{align*}
Define a function $g\in \FF$ as follows: $g|_{I_n}:=u|_{I_n}$ for $1\leq n\leq N$ and $g:=0$ elsewhere. Then $g$ is continuous on $I_n$ for $1\leq n\leq N$ and
\[
	\EE_1(u_1-g,u_1-g)<2\epsilon.
\]
Mimicking the proof of \cite[Lemma~3.14]{LY16} and referring to Remark~\ref{RM27} below, we can construct a function $g_\epsilon \in \FF\cap C_c(I)$ such that $\EE_1(g_\epsilon-g,g_\epsilon-g)<\epsilon$. This implies
\[
	\EE_1(g_\epsilon-u_1,g_\epsilon-u_1)<6\epsilon.
\]

\begin{remark}\label{RM27}
Similar to \cite[Lemma~3.14]{LY16}, the discontinuous points of $g$ are contained in the set of endpoints $\{a_n,b_n: 1\leq n\leq N\}\setminus \{l,r\}$. Consider the right discontinuous endpoint $c:=b_n$, we can adjust $g$ through the method in \cite[Lemma~3.14]{LY16} except for one possible additional case: $c\in I_n, g(c)\neq 0$ and $(c,c+\delta)\subset \left(\cup_{n\geq 1}I_n\right)^c$ for some $\delta>0$. Note that in the context of \cite{LY16}, this is impossible since $m( \left(\cup_{n\geq 1}I_n\right)^c)=0$ by \cite[Theorem~3.3]{LY16}. However, we do not have this restriction in this theorem. This additional case is easy to solve. Without loss of generality, assume that $g(c)=1$. Notice that $g\equiv 0$ on $(c,c+\delta)$. Take a constant $\delta' <\delta$ with $m((c,c+\delta'))<\epsilon$ and set $$\phi(x):=\begin{cases} 1-(1/{\delta'})\cdot (x-c),& c\leq x\leq c+\delta'\\
 0,&\text{elsewhere}.\end{cases}$$ Then $g+\phi$ is continuous at $c$ and $$\EE_1(\phi,\phi)=\int_I \phi^2dm<\epsilon.$$
\end{remark}

Next we treat $u_2$. Since $u_2\in L^2(I,m)$, it follows from \cite[Exercise~2.13]{CF12} that there exists a function $h\in \FF\cap C_c(I)$ such that $\|h-u_2\|^2_{L^2(I,m)}<\epsilon$. The key step to complete the proof is to 
construct a function $h_{\epsilon}$ through $h$ so that it satisfies not only $$h_{\epsilon}\in \FF\cap C_c(I), \quad \|h_{\epsilon}-u_2\|^2_{L^2(I,m)}<\tilde{C}\epsilon$$ but also $\EE(h_{\epsilon},h_{\epsilon})<\tilde{C}\epsilon$ with some constant $\tilde{C}$. 
The construction of $h_\epsilon$ is given in the following lemma.

\begin{lemma}\label{LM39}
Assume $u_2\in L^2(I,m)$ with $u_2|_{\cup_{n\geq 1} I_n}=0$ and fix a small constant $\epsilon>0$. Then there exists a function $h_\epsilon\in \FF\cap C_c(I)$ such that
\[
\|h_{\epsilon}-u_2\|^2_{L^2(I,m)}<18\epsilon, \quad \EE(h_{\epsilon},h_{\epsilon})<6\epsilon.
\]
\end{lemma}

\begin{proof}
Let $h$ be above, i.e. $h\in \FF\cap C_c(I)$ such that $\|h-u_2\|^2_{L^2(I,m)}<\epsilon$. Since $$\EE(h,h)=\sum_{n\geq 1}\EE^{(\ss_n)}(h|_{I_n}, h|_{I_n})<\infty,$$ there exists an integer $M$ such that $\sum_{n> M}\EE^{(\ss_n)}(h|_{I_n}, h|_{I_n})<\epsilon$. Recall that $e_n$ is a fixed point in $(a_n,b_n)$.
To attain $h_\epsilon$, we shall change the values of $h$ near the intervals $\{I_n: 1\leq n\leq M\}$.

We first classify the right endpoints $\{b_n: 1\leq n\leq M\}$ as follows:
\begin{itemize}
\item $b_n\in \mathcal{R}_1$ if $b_n=a_j$ for some $1\leq j\leq M$;
\item $b_n\in \mathcal{R}_2$ if $b_n\notin \mathcal{R}_1$ and $b_n\notin I_n$;
\item $b_n\in \mathcal{R}_3$ if $b_n\notin \mathcal{R}_1$ and $b_n\in I_n$.
\end{itemize}
Similarly we can classify the left endpoints $\{a_n: 1\leq n\leq M\}$ as $\mathcal{L}_1, \mathcal{L}_2$ and $\mathcal{L}_3$. For each $1\leq n\leq M$, we shall take a small constant $\varepsilon_n\geq 0$ and define a continuous function $h^+_n$ on $[e_n, b_n+\varepsilon_n)$ with $h^+_n(e_n)=0$. The construction is as follows.
\begin{itemize}
\item[(1)] $b_n\in \mathcal{R}_1$. Take $\varepsilon_n=0$.
\begin{itemize}
\item[(1a)] $b_n\in I_n$ or $a_j\in I_j$: define $h^+_n(x)=0$ for any $x\in [e_n, b_n]$.
\item[(1b)] $b_n\notin I_n$ and $a_j\notin I_j$. In this case, $\ss_n(b_n)=\infty$. Take a constant $\delta>0$. Mimicking (1a) of \cite[Lemma~3.14]{LY16}, we can take another constant $0<\delta'<\delta$ and define $h_n^+$ on $[e_n,b_n)$ as follows:
\begin{equation}\label{EQ2HXE}
	h_n^+(x):=\begin{cases} 0,& e_n\leq x\leq b_n-\delta,\\
	h(b_n)\cdot \displaystyle{\frac{\ss_n(x)-\ss_n(b_n-\delta)}{\ss_n(b_n-\delta')-\ss_n(b_n-\delta)}},& b_n-\delta<x<b_n-\delta',\\
	h(b_n),& b_n-\delta'\leq x\leq b_n.
	\end{cases}
\end{equation}
Regard $h^+_n$ as a function on $I_n$ by letting $h_n^+(x)=0$ on $\langle a_n,e_n)$. Then clearly $h^+_n\in \FF^{(\ss_n)}$ and let $\delta>0$ be small enough and then $\delta'<\delta$ be small enough, we have
\[
\begin{aligned}
	\EE^{(\ss_n)}_1&(h^+_n, h^+_n)\\
	& \leq h(b_n)^2\cdot m((b_n-\delta, b_n)) + \frac{1}{2}h(b_n)^2\frac{1}{\ss_n(b_n-\delta')-\ss_n(b_n-\delta)} \\
	& < \frac{\epsilon}{M}.
\end{aligned}
\]
\end{itemize}
\item[(2)] $b_n\in \mathcal{R}_2$. Take $\varepsilon_n=0$.
\begin{itemize}
\item[(2a)] $b_n\notin I$, in other words, $b_n=r\notin I$: define $h^+_n(x)=0$ for any $x\in [e_n, b_n)$.
\item[(2b)] $b_n\in I$: define $h^+_n$ as \eqref{EQ2HXE} and we have $\EE^{(\ss_n)}_1(h^+_n, h^+_n)< \epsilon/M$.
\end{itemize}
\item[(3)] $b_n\in \mathcal{R}_3$ and $b_n\neq r$.
\begin{itemize}
\item[(3a)] $b_n=a_k$ for some $k>M$. Take $\varepsilon_n=b_k-b_n$.

In this case, $a_k\notin I_k$ and $\ss_k(a_k)=-\infty$. Let $h_n^+(x):=0$ for any $x\in [e_n,b_n]$ and mimicking the case (1b), take $0<\delta'<\delta$ small enough and define
\begin{equation}\label{EQ2HEX}
	h_n^+(x):= \begin{cases} 0, &\quad a_k\leq x\leq a_k+\delta',\\
	h(a_k+\delta)\cdot \displaystyle{\frac{\ss_k(x)-\ss_k(a_k+\delta')}{\ss_n(a_k+\delta)-\ss_n(a_k+\delta')}},&\quad a_k+\delta'<x<a_k+\delta,\\
	h(x),&\quad a_k+\delta\leq x<b_k .
	\end{cases}
\end{equation}
When $\delta, \delta'$ are small enough, we have
\begin{equation}\label{EQ3ESK}
\EE_1^{(\ss_k)}((h_n^+-h)|_{I_k},(h_n^+-h)|_{I_k})<\frac{\epsilon}{M}.
\end{equation}
\item[(3b)] For any $\varepsilon>0$, $(b_n, b_n+\varepsilon)$ contains at least one non-closed interval $I_k$ with $k>M$.
In this case, without loss of generality, assume
\begin{equation}\label{EQ3BNB}
	\int_{(b_n, b_n+\epsilon)}h(x)^2m(dx)<\epsilon/M,
\end{equation}
and $I_k=(a_k, b_k\rangle \subset (b_n, b_n+\varepsilon)$ with $k>M$.

Take $\varepsilon_n=b_k-b_n$. Define $h^+_n(x):=0$ for any $x\in [e_n, a_k)$ and for $x\in [a_k, b_k)$, define $h^+_n(x)$ as \eqref{EQ2HEX}. Similarly, \eqref{EQ3ESK} holds.
\item[(3c)] For any $\varepsilon>0$, $(b_n, b_n+\varepsilon)$ contains at least one interval $I_j$ with $j>M$ and all of them are closed.
Without loss of generality, assume $b_n+\varepsilon\notin \cup_{j\geq 1} \mathring{I}_j$ and $\varepsilon$ is small enough so that \eqref{EQ3BNB} holds.

Take $\varepsilon_n=\varepsilon$ and set $h_n^+(x)=0$ for $x\in [e_n,b_n]$. By mimicking the case (2) of \cite[Lemma~3.14]{LY16} and \cite[Remark~3.15]{LY16} we can construct a Cantor-type function $h_n^+$ on $[b_n,b_n+\varepsilon_n]$ such that $h_n^+(b_n)=0, h_n^+(b_n+\varepsilon_n)=h(b_n+\varepsilon_n)$ and $h_n^+$ is a constant on every interval $I_j$, which is contained in $(b_n,b_n+\varepsilon)$. (Although $\bigcup_{n\geq 1}I_n$ is not necessarily dense in $I$, this construction is not difficult to realize by small modifications of \cite[Remark~3.15]{LY16}). Without loss of generality, we also have
\begin{equation}\label{EQ3BNBN}
\int_{(b_n, b_n+\varepsilon_n)} h_n^+(x)^2m(dx)<\frac{\epsilon}{M}.
\end{equation}
\item[(3d)] For some $\varepsilon>0$, $(b_n,b_n+\varepsilon)\subset \left(\bigcup_{n\geq 1}I_n\right)^c$.

Take $\varepsilon_n<\varepsilon$ small enough and set
\[
	h_n^+(x):=\begin{cases}
		0,&\quad x\in [e_n,b_n],\\
		h(b_n+\varepsilon_n)\cdot \displaystyle{\frac{x-b_n}{\varepsilon_n}},&\quad b_n<x<b_n+\varepsilon_n
	\end{cases}
\]
Without loss of generality, we have \eqref{EQ3BNB} and \eqref{EQ3BNBN} hold.
\end{itemize}
\item[(4)] $b_n\in \mathcal{R}_3$ and $b_n=r$.

Take $\varepsilon_n=0$ and define $h^+_n(x):=0$ for $x\in [e_n, b_n]$.
\end{itemize}

Mimicking the procedures above, we can also take another small constant $\varepsilon'_n\geq 0$ and define a continuous function $h^-_n$ on $(a_n-\varepsilon'_n, e_n]$ with $h^-_n(e_n)=0$. Now we can define the function $h_\epsilon$ as follows:
\[
	h_\epsilon(x):=\begin{cases}
		h^+_n(x),&\quad x\in [e_n,b_n+\varepsilon_n),\; 1\leq n\leq M,\\
		h^-_n(x), &\quad x\in (a_n-\varepsilon'_n, e_n],\; 1\leq n\leq M, \\
		0, &\quad x=b_n=a_j \text{ for case (1a)}, \\
		h(x), &\quad \text{others}.
	\end{cases}
\]
From the construction of $h^\pm_n$, we can easily see that $h_\epsilon\in C_c(I)$.

For any function $g$ on $I$ such that $2\EE(g,g)=\sum_{n\geq 1}\int_{I_n}\left(dg/d\ss_n\right)^2d\ss_n<\infty$ and any subset $J\subset I$, we regard $g|_J$ as a function on $I$ by letting $g|_J(x)=0$ for $x\in I\setminus J$ and set
\[
	\EE(g|_J, g|_J):=\frac{1}{2}\sum_{n\geq 1}\int_{I_n\cap J}\left(\frac{dg}{d\ss_n}\right)^2d\ss_n.
\]
Let $A$ be the set of all the endpoints $\{b_n: 1\leq n\leq M\}$ in case (1a) and set
\[
	G:=\left(\bigcup_{n=1}^M(a_n-\varepsilon'_n, b_n+\varepsilon_n) \right)^c\setminus A.
\]
Then we have
\[
\begin{aligned}
\EE(h_\epsilon, h_\epsilon)&=\sum_{n=1}^M\EE(h^+_n, h^+_n)+\sum_{n=1}^M \EE(h^-_n, h^-_n)+ \EE(h|_G, h|_G).
\end{aligned}
\]
For the cases (1), (2), (3c), (3d) and (4), clearly $\EE(h^+_n, h^+_n)<\epsilon/M$ and for the cases (3a) and (3b), it follows from \eqref{EQ3ESK} that
\[
\EE(h^+_n, h^+_n)\leq \frac{2\epsilon}{M}+2\EE(h|_{I_k}, h|_{I_k}),
\]
where $I_k$ is the interval given in (3a) and (3b). The similar estimate holds for $\EE(h^-_n,h^-_n)$. This implies $h_\epsilon\in \FF$ and
\[
\EE(h_\epsilon, h_\epsilon)\leq \frac{2\epsilon}{M}\cdot M+ \frac{2\epsilon}{M}\cdot M+ 2\sum_{k>M} \EE^{(\ss_k)}(h|_{I_k}, h|_{I_k})<6\epsilon.
\]

On the other hand,
\[
\begin{aligned}
\int_I& |h_\epsilon-u_2|^2dm\\
&=\sum_{n= 1}^M \int_{[e_n, b_n+\varepsilon_n)} \left(h^+_n-u_2\right)^2dm+\sum_{n= 1}^M \int_{(a_n-\varepsilon'_n, e_n]} \left(h^-_n-u_2\right)^2dm\\
&\qquad \qquad \qquad \qquad \qquad \qquad +\int_G \left(h-u_2\right)^2dm.
\end{aligned}
\]
For the cases (1), (2) and (4), clearly $\int_{[e_n, b_n+\varepsilon_n)} \left(h^+_n-u_2\right)^2dm<\epsilon/M$. For case (3a), it follows from \eqref{EQ3ESK} that
\[
\begin{aligned}
\int_{[e_n, b_n+\varepsilon_n)} &\left(h^+_n-u_2\right)^2dm \\ &=\int_{(b_n, b_n+\varepsilon_n)}\left(h^+_n-u_2\right)^2dm  \\
&\leq 2\int_{(b_n, b_n+\varepsilon_n)}\left(h^+_n-h\right)^2dm+2\int_{(b_n, b_n+\varepsilon_n)}\left(h-u_2\right)^2dm \\
&<\frac{2\epsilon}{M} +2\int_{(b_n, b_n+\varepsilon_n)}\left(h-u_2\right)^2dm.
\end{aligned}
\]
For the case (3b), we can deduce from \eqref{EQ3ESK} and \eqref{EQ3BNB} that
\[
\int_{[e_n, b_n+\varepsilon_n)} \left(h^+_n-u_2\right)^2dm<\frac{4\epsilon}{M} +2\int_{(b_n, b_n+\varepsilon_n)}\left(h-u_2\right)^2dm.
\]
For the cases (3c) and (3d), from \eqref{EQ3BNB} and \eqref{EQ3BNBN}, we can obtain that
\[
\int_{[e_n, b_n+\varepsilon_n)} \left(h^+_n-u_2\right)^2dm<\frac{8\epsilon}{M} +2\int_{(b_n, b_n+\varepsilon_n)}\left(h-u_2\right)^2dm.
\]
The term $\int_{(a_n-\varepsilon'_n, e_n]} \left(h^-_n-u_2\right)^2dm$ can be estimated similarly.
Therefore, we can conclude that
\[
\int_I |h_\epsilon-u_2|^2dm<\frac{8\epsilon}{M}\cdot M +\frac{8\epsilon}{M}+2\int_I (h-u_2)^2dm<18\epsilon.
\]
That completes the proof.
\end{proof}

Finally, let $h_\epsilon$ be the function given in Lemma~\ref{LM39}. We have
\[ \EE_1(h_\epsilon-u_2,h_\epsilon-u_2)=\EE(h_\epsilon,h_\epsilon)+\|h_\epsilon-u_2\|_{L^2(I,m)}^2<24\epsilon.
\]
and then, by setting $u_\epsilon:=g_\epsilon+h_\epsilon\in \FF\cap C_c(I)$, it holds that
\[
\begin{aligned}
\EE_1(u_\epsilon-u, u_\epsilon-u)&\leq 2\left(\EE_1(g_\epsilon-u_1,g_\epsilon-u_1)+\EE_1(h_\epsilon-u_2,h_\epsilon-u_2) \right) \\
&<60\epsilon.
\end{aligned}
\]
That completes the proof of sufficiency.



\subsection{Effective intervals}\label{SEC231}
Since the diffusion $X$ in Theorem~\ref{THM21} has non-trivial motions only in each interval $I_n$ ($n\geq 1$), we call the interval $I_n$ with the adapted scale function $\ss_n\in \SS_\infty(I_n)$ an \emph{effective interval} of $X$ or $(\EE,\FF)$. All effective intervals of $X$ (or $(\EE,\FF)$) are $\{I_n:n\geq 1\}$ or $\{(I_n,\ss_n):n\geq 1\}$ if the scale functions are clarified. Clearly, $X$ or $(\EE,\FF)$ is determined by the effective intervals $\{(I_n,\ss_n):n\geq 1\}$ if the state space $I$ and the symmetric measure $m$ are given.

We give more remarks for the effective interval $(I_n,\ss_n)$. Except for the case $I_n=\langle a_n,b_n\rangle$ shares the same open endpoint of $I$ (i.e. (L0) or (R0)), the relation between $I_n$ and $\ss_n$ is simple:
\begin{equation}\label{EQ2ANI}
	a_n\in I_n \Leftrightarrow \ss_n(a_n)>-\infty, \quad b_n\in I_n \Leftrightarrow \ss_n(b_n)<\infty.
\end{equation}
This implies that $a_n$ (resp. $b_n$) is a reflected endpoint if $a_n\in I_n$ (resp. $b_n\in I_n$) and an unapproachable endpoint if $a_n\notin I_n$ (resp. $b_n\notin I_n$) relative to $X$. For the case $I_n$ has a common open endpoint of $I$, i.e. $I_n=(l, b_n\rangle$ and $I=(l, r\rangle$ (resp. $I_n=\langle a_n,r)$ and $I=\langle l,r)$), then $\ss_n(l)>-\infty$ (resp. $\ss_n(r)<\infty$) is allowed but $l$ (resp. $r$) must be an absorbing endpoint relative to $X$ (see \eqref{EQ2FSU}).

The example below gives us an intuitive understanding to the necessity of \eqref{EQ2ANI} for the regularity of $(\EE,\FF)$.

\begin{example}\label{EXA21}
We first consider the absorbing Browning motion $X$ on $I_1=(0,1)$. Precisely, it is associated with the regular Dirichlet form $(\frac{1}{2}\mathbf{D}, H^1_0(I_1))$ on $L^2(I_1)$, where
\[
\begin{aligned}
	&H^1_0(I_1)=\big\{u\in L^2(I_1): u\text{ is absolutely continuous on }I_1,\\  &\qquad \qquad \qquad\qquad \qquad \qquad\qquad \qquad   u'\in L^2(I_1),\; u(0)=u(1)=0\big\}, \\
	&\frac{1}{2}\mathbf{D}(u,v)=\frac{1}{2}\int_{I_1} u'(x)v'(x)dx,\quad u,v\in H^1_0(I_1).
\end{aligned}
\]
Let $(B_t)_{t\geq 0}$ be the standard Brownian motion on $\mathbb{R}$ and set $\zeta:=\sigma_0\wedge \sigma_1$ where $\sigma_0$ and $\sigma_1$ are the hitting times of $\{0\}$ and $\{1\}$ relative to $(B_t)_{t\geq 0}$. Clearly, $\zeta$ is the lifetime of $X$ and
\begin{equation}\label{EQ1XTB}
	X_t= \left\lbrace \begin{aligned}& B_t,\quad t<\zeta,  \\
	&\partial, \quad t\geq \zeta, \end{aligned}  \right.
\end{equation}
where $(0,1)\cup \partial $ is the one-point compactification of $(0,1)$. As a rough observation, we may say $\partial$ coincides with $\{0,1\}$. Thus $t\mapsto X_t$ is continuous on $[0,\infty)$ and $X$ is a so-called Hunt process on $(0,1)$.

Now let $I=[0,1]$ and $m$ be the Lebesgue measure on $I$.
We treat the quadratic form $(\frac{1}{2}\mathbf{D}, H^1_0(I_1))$ on $L^2(I)$. Since $|\{0,1\}|=0$ and $u(0)=u(1)=0$ for any $u\in H^1_0(I_1)$, it is really a Dirichlet form on $L^2(I)$ but not regular any more (notice that $H^1_0(I_1) \cap C(I)$ cannot separate $0$ and $1$). Nevertheless, Remark~\ref{RM22} after this example  assures that $(\frac{1}{2}\mathbf{D}, H^1_0(I_1))$ is quasi-regular on $L^2(I)$ and its associated Markov process $X^*$ is only a standard process on $I$. Let us explain why $X^*$ is not a Hunt process. In fact, $X^*$ still has an expression similar to \eqref{EQ1XTB}:
\[
	X^*_t= \begin{cases} B_t,&\quad t<\zeta,  \\
	\partial, &\quad t\geq \zeta, \end{cases}
\]
whereas $\partial$ is now an isolated point of $I$. We have $X^*_{\zeta-}=0$ or $1$ but $X^*_\zeta=\partial$ ($\neq 0$ or $1$). Thus $t\mapsto X^*_t$ is only continuous on $[0, \zeta)$ and $X^*$ is not a Hunt process on $I$.

Note that $I_1=(0,1)$ may be treated as an effective interval of $(\frac{1}{2}\mathbf{D}, H^1_0(I_1))$ on $L^2(I)$. The scale function on $I_1$ is the natural scale, i.e. $\ss_1(x)=x$. Clearly, $I_1$ and $\ss_1$ do not satisfy \eqref{EQ2ANI}.
\end{example}
\begin{remark}\label{RM22}
Fix a quasi-regular Dirichlet form $(\EE,\FF)$ on $L^2(E,m)$ (Cf. \cite{MR92}). Assume $E$ is a subset of a topological space $E^*$ and extend $m$ to a new measure $m^*$ on $E^*$ by $m^*(E^*\setminus E)=0$. Then $(\EE,\FF)$ is still a quasi-regular Dirichlet form on $L^2(E^*,m^*)$ and $E^*\setminus E$ is an $\EE$-polar set relative to its associated Markov process.
\end{remark}

The following example including the $d$-Bessel process for an integer $d\geq 2$ tells us a surprising fact: the presence of additional exceptional set might not break the regularity of the Dirichlet form. Notice that Example~\ref{EXA21} is the contrary of this fact.

\begin{example}\label{EXA23}
Set $I_1=(0,\infty)$ and consider an $m$-symmetric diffusion process $Y$ on $I_1$ with a scale function $\ss_1$ satisfying $\ss_1(0)=-\infty$. For example, for some integer $d\geq 2$, let
\begin{equation}\label{EQ2SXL}
	\ss_1(x):=\left\lbrace \begin{aligned} &\log x,\quad d=2,\\
	&\frac{{x^{2-d}-1}}{2-d},\quad d\geq 3.  \end{aligned} \right.
\end{equation}
When $d\geq 3$, we have {$\ss_1(\infty)=-1/(2-d)<\infty$} and thus (R0) happens when $m([1,\infty))<\infty$.
The associated Dirichlet form of $Y$ on $L^2(I_1,m)$ is given by
\begin{equation}\label{EQ1FUL}
\begin{aligned}
	&\FF=\left\{u\in L^2(I_1,m): u\ll \ss_1,\ \frac{du}{d\ss_1}\in L^2(I_1,d\ss_1), \ u(\infty)=0\text{ whenever (R0)}\right\}, \\
	&\EE(u,v)=\frac{1}{2}\int_{I_1} \frac{du}{d\ss_1}\frac{dv}{d\ss_1}d\ss_1,\quad u,v \in \FF.
\end{aligned}\end{equation}
Note that $(\EE,\FF)$ is a regular Dirichlet form on $L^2(I_1, m)$ and $Y$ is a regular diffusion on $I_1$ (Cf. \cite[Proposition~2.2.10]{CF12}).
Particularly, when
\begin{equation}\label{EQ2MXX}
	m(dx)=x^{d-1}dx
\end{equation}
and $\ss_1$ is given by \eqref{EQ2SXL}, $Y$ is an equivalent version of the so-called $d$-Bessel process on $(0,\infty)$.
It is well-known that the $d$-Bessel process (i.e. $(|B|_t)_{t\geq 0}$ for a $d$-dimensional Brownian motion $(B_t)_{t\geq 0}$) is usually defined as a regular diffusion on $[0, \infty)$, where $0$ is the entrance boundary.
So as in Example~\ref{EXA21}, we put $(\EE,\FF)$ on the state space
\[
I=[0,\infty).
\]
Further assume that $m(0+)<\infty$ and then $m$ can be naturally extended to a Radon measure, which we still denote it by $m$, on $I$ (such as \eqref{EQ2MXX}). Clearly, $(\EE,\FF)$ is still a Dirichlet form on $L^2(I, m)$.
From Theorem~\ref{THM21}, we can conclude that $(\EE,\FF)$ is regular not only on $L^2(I_1,m)$ but also on $L^2(I,m)$.
Particularly, the interval $I_1$ with a scale function $\ss_1$ on it is an effective interval of regular Dirichlet form $(\EE,\FF)$ on $L^2(I,m)$ and $\{0\}$ is an $\EE$-exceptional set. As an example, the associated Dirichlet form of $d$-Bessel process is regular on $L^2((0,\infty),m)$ and $L^2([0,\infty),m)$ at the same time.

Let us briefly explain why $(\EE,\FF)$ is possibly regular on $L^2(I,m)$.  Denote the associated diffusion process of $(\EE,\FF)$ on $L^2(I,m)$ by $Y^*$. Then $Y^*$ is a Hunt process on $I$. In fact, $Y^*$ coincides with $Y$ on $I$ and $\{0\}$ is an $\EE$-polar set relative to $Y^*$. The difference to $X^*$ in Example~\ref{EXA21} is that the boundary $\{0,1\}$ is approachable for $X$ or $X^*$ but $\{0\}$ is not approachable for $Y$ or $Y^*$ since $\ss_1(0)=-\infty$. In other words, $Y$ is always staying in $I$ before leaving from the boundary `$\infty$', no matter $\{0\}$ being or not being contained in the state space.
\end{example}

\subsection{When  $I=\mathbb{R}$}

In this section, we shall discuss the special case of Theorem~\ref{THM21} for $I=\mathbb{R}$ which is of special importance.
Let $J=\langle a, b\rangle \subset \mathbb{R}$ be an interval and $\SS(J)$ be the class of scale functions on $J$ as in \S\ref{SEC21}. Since $I=\mathbb{R}$, the adapted conditions \textbf{(A)} and \textbf{(B)} for a scale function $\ss\in \SS(J)$ could be written as
\begin{itemize}
\item[\textbf{($\textbf{A}_R$)}] $a\in J$ if and only if $a+\ss(a)>-\infty$;
\item[\textbf{($\textbf{B}_R$)}] $b\in J$ if and only if $b+\ss(b)<\infty$.
\end{itemize}
Then set
\begin{equation}\label{EQ2SJS2}
\SS_\infty(J):=\left\{\ss\in \SS(J): \ss \text{ satisfies }\textbf{(}\textbf{A}_R\textbf{)} \text{ and }\textbf{(}\textbf{B}_R\textbf{)}\right\}.
\end{equation}
Moreover, take $\ss\in \SS_\infty(J)$. The allowable situations (L0) and (R0) (roughly speaking, absorbing at $-\infty$ or $\infty$) are given as follows:
\begin{itemize}
\item[($\text{L}_R$)] $a=-\infty$, $\ss(-\infty)>-\infty$ and $m\left((-\infty, 0]\right)<\infty$;
\item[($\text{R}_R$)] $b=\infty$, $\ss(\infty)<\infty$ and $m\left([0, \infty)\right)<\infty$.
\end{itemize}
Similarly, define a quadratic form $(\EE^{(\ss)}, \FF^{(\ss)})$ on $L^2(J,m_J)$ as follows:
\begin{equation}\label{EQ2FSU2}
\begin{aligned}
&\FF^{(\ss)}:=\bigg\{u\in L^2(J,m_J): u\ll \ss, \; \frac{du}{d\ss}\in L^2(J,d\ss); \\
 &\qquad \qquad\qquad  u(a)=0\text{ (resp. }u(b)=0\text{) whenever (}\text{L}_R\text{) (resp. (}\text{R}_R\text{))} \bigg\},  \\
& \EE^{(\ss)}(u,v):=\frac{1}{2}\int_J \frac{du}{d\ss}\frac{dv}{d\ss}d\ss,\quad u,v\in \FF^{(\ss)}.
\end{aligned}
\end{equation}
The following corollary is the special case of Theorem~\ref{THM21} for $I=\mathbb{R}$.

\begin{corollary}\label{COR212}
Let $m$ be a fully supported Radon measure on $\mathbb{R}$. Then $(\EE, \FF)$ is a regular and strongly local Dirichlet form on $L^2(\mathbb{R},m)$ if and only if there exist a set of at most countable disjoint intervals $\{I_n=\langle a_n,b_n\rangle: I_n\subset \mathbb{R}, n\geq 1\}$ and a scale function $\ss_n\in \SS_\infty(I_n)$ for each $n\geq 1$ such that
\begin{equation}\label{EQ2FULR}
\begin{aligned}
	&\FF=\left\{u\in L^2(\mathbb{R}, m): u|_{I_n}\in \FF^{(\ss_n)}, \sum_{n\geq 1}\EE^{(\ss_n)}(u|_{I_n}, u|_{I_n})<\infty  \right\},  \\
	&\EE(u,v)=\sum_{n\geq 1}\EE^{(\ss_n)}(u|_{I_n}, v|_{I_n}),\quad u,v \in \FF,
\end{aligned}
\end{equation}
where for each $n\geq 1$, $(\EE^{(\ss_n)}, \FF^{(\ss_n)})$ is given by \eqref{EQ2FSU2} with the scale function $\ss_n$ on $I_n$.  Moreover, the intervals $\{I_n: n\geq 1\}$ and scale functions $\{\ss_n:n\geq 1\}$ are uniquely determined, if the difference of order is ignored.
\end{corollary}

We end this section with a remark for the connection of Dirichlet forms in \eqref{EQ2FUL} and \eqref{EQ2FULR}. Let $(\EE,\FF)$ be the Dirichlet form \eqref{EQ2FUL} on $L^2(I,m)$, where $I=\langle l, r\rangle$ is an interval of $\mathbb{R}$, and $X$ be its associated diffusion process on $I$.
Further let $(\EE^*,\FF^*)$ be another regular and strongly local Dirichlet form on $L^2(I^*,m^*)$ associated with the diffusion $X^*$ where $I^*$ is another interval and $m^*$ is a fully supported Radon measure on $I^*$. The effective intervals of $(\EE^*,\FF^*)$ are denoted by $\{(I^*_n, \ss^*_n)\}$. We say $(\EE,\FF)$ and $(\EE^*,\FF^*)$ {are \emph{essentially equivalent}} if
\[
\left\{(I_n,\ss_n):n\geq 1\right\}=\left\{(I^*_n,\ss^*_n):n\geq 1\right\},\quad m|_{\cup_{n\geq 1}I_n}=m^*|_{\cup_{n\geq 1} I^*_n}.
\]
This condition indicates the $X^*$ coincides with $X$ on each common effective interval (notice that outside the effective intervals, the diffusions have no motions).
When $l$ ($>-\infty$) is a closed endpoint of $I$, i.e. $I=[l,r\rangle$, we can extend $m$ to a measure $m^*$ on $(-\infty, r\rangle$ by
\[
	m^*|_{(-\infty, l)}=dx|_{(-\infty, l)}, \quad m^*|_I:=m|_I.
\]
Then $m^*$ is a fully supported Radon measure on $(-\infty, r\rangle$ and $(\EE,\FF)$ induces a new Dirichlet form on $L^2((-\infty, r\rangle, m^*)$:
\begin{equation}\label{EQ2FUL2}
\begin{aligned}
	&\FF^*:=\{u\in L^2((-\infty,r\rangle, m^*): u|_I\in \FF\},\\
	&\EE^*(u,v):= \EE(u|_I, v|_I),\quad u,v\in \FF^*.
\end{aligned}\end{equation}
Clearly, $(\EE^*,\FF^*)$ is regular by Theorem~\ref{THM21} and essentially equivalent to $(\EE,\FF)$. When $I=(l,r\rangle$, $m(l+)<\infty$ and (L0) does not appear, the same trick is also valid. Therefore, we can always find an essentially equivalent Dirichlet form of $(\EE,\FF)$ on an extended open interval. Particularly, when $m$ is a finite measure on $I$ and (L0) and (R0) do not appear, this extended open interval can be taken as $\mathbb{R}$ and $(\EE,\FF)$ is essentially equivalent to a Dirichlet form $(\EE^*,\FF^*)$ on $L^2(\mathbb{R},m^*)$, which can be characterized by Corollary~\ref{COR212}. Roughly speaking, without loss of generality, we may always assume $I=\mathbb{R}$ when considering the questions about a regular and strongly local Dirichlet form associated with a symmetric {one-dimensional} diffusion.

\section{When $\FF$ contains smooth functions}
\label{SEC3}

A typical way to construct a regular Dirichlet form is to start from a closable Markovian bilinear form on the domain $C_c^\infty$ and then the closure of $C_c^\infty$ with this form is a regular Dirichlet form, see \cite[\S3]{FOT11}. Let $(\EE,\FF)$ be the regular Dirichlet form given in Theorem~\ref{THM21} and assume that
\begin{equation}\label{EQ3CCI}
	C_c^\infty(I)\subset \FF.
\end{equation}
What we are concerned of in this section is the representation of the closure of $C_c^\infty(I)$ in $(\EE,\FF)$.

\subsection{Regular Dirichlet subspaces}

The theory of regular Dirichlet subspaces will be heavily used in this section.
Given two regular Dirichlet forms $(\EE^1,\FF^1)$ and $(\EE^2,\FF^2)$ on $L^2(E,m)$, where $E$ is a locally compact separable metric space and $m$ is a fully supported Radon measure on $E$, $(\EE^1,\FF^1)$ is called a regular Dirichlet subspace of $(\EE^2,\FF^2)$ if
\[
	\FF^1\subset \FF^2,\quad \EE^2(u,v)=\EE^1(u,v),\quad u,v\in \FF^1.
\]
Under the assumption \eqref{EQ3CCI}, denote the closure of $C_c^\infty(I)$ in $(\EE,\FF)$ by $\bar{\FF}$ and define
\[
	\bar{\EE}(u,v):=\EE(u,v),\quad u,v\in \bar{\FF}.
\]
Then clearly $(\bar{\EE},\bar{\FF})$ is a regular Dirichlet subspace of $(\EE,\FF)$.

Basic properties about the regular Dirichlet subspaces are presented in \S2 of \cite{LY16}. Particularly, let $(\EE,\FF)$ be a regular, irreducible and  strongly local Dirichlet form on $L^2(J,m_J)$, where $J$ is an open interval. Denote its scale function by $\ss$. The following result is taken from \cite{FFY05} and \cite{FHY10} but we need a little more explanation for its proof.

\begin{lemma}\label{LM31}
Let $(\EE,\FF)=(\EE^{(\ss)},\FF^{(\ss)})$ be given by \eqref{EQ2FSU}, in which $I=J$.  Then $(\tilde{\EE},\tilde{\FF})$ is a regular Dirichlet subspace of $(\EE,\FF)$ if and only if $(\tilde{\EE},\tilde{\FF})=(\EE^{(\tilde{\ss})}, \FF^{(\tilde{\ss})})$, where the scale function $\tilde{\ss}$ satisfies
\begin{equation}\label{EQ3DSD}
	d\tilde{\ss}\ll d\ss,\quad \frac{d\tilde{\ss}}{d\ss}=0\text{ or }1,\quad d\ss\text{-a.e.}
\end{equation}
\end{lemma}
\begin{proof}
Before applying Theorem~4.1 of \cite{FHY10}, we need to explain the following fact: if $(\tilde{\EE},\tilde{\FF})$ is a regular Dirichlet subspace of $(\EE,\FF)$, then $(\tilde{\EE},\tilde{\FF})$ is strongly local and its associated diffusion process $\tilde{X}$ is irreducible in the sense of \eqref{EQ2PXY}. The strongly local property of $(\tilde{\EE},\tilde{\FF})$ follows from Theorem~2.1 of \cite{LY14}. Propositions 2.2 and 2.3 of \cite{LY16} imply the irreducibility of $\tilde{X}$. The rest of the proof is clear.
\end{proof}

\subsection{Irreducible case}\label{SEC32}

The irreducible case of $I=J$, where $J$ is an open interval, is studied in \cite{FL15}. Let $(\EE,\FF)$ be the Dirichlet form on $L^2(J,m_J)$ in Lemma~\ref{LM31} and $\ss$ its scale function. The inverse function of $\ss$ is denoted by $\tt$, i.e. $\tt=\ss^{-1}$. Then \eqref{EQ3CCI} is equivalent to the condition (Cf. \cite[Lemma~2.1]{FL15})
\begin{equation}\label{EQ3TIA}
	\tt\text{ is absolutely continuous, and }\tt'\in L^2_\text{loc}.
\end{equation}
This assumption does not necessarily imply $\ss$ is absolutely continuous, but we may write the Lebesgue decomposition of $d\ss$ as follows:
\[
	d\ss=g(x)dx+\kappa,\quad \kappa \perp dx,
\]
where $g\in L^1_\mathrm{loc}(J)$. Note that $g>0$ a.e. In fact, let $Z_g:=\{x\in J\setminus H: g(x)=0\}$, where $H$ is the support of $\kappa$. Clearly, $|H|=0$ and $d\ss(Z_g)=\int_{Z_g}g(x)dx=0$. Thus we must have $|Z_g|=0$ since $dx\ll d\ss$. From this fact, we can define the absolutely continuous part of $\ss$ as
\[
\bar{\ss}(x):=\int_e^x g(y)dy,
\]
where $e$ is a fixed point in $J$. Clearly, $\bar{\ss}$ is a scale function on $J$, i.e. a strictly increasing and continuous function on $J$.

Assume \eqref{EQ3CCI} and denote the Dirichlet form induced by the closure of $C_c^\infty(J)$ by $(\bar{\EE},\bar{\FF})$. From Lemma~\ref{LM31}, we know that  $(\bar{\EE},\bar{\FF})$ corresponds to a diffusion process with some scale function $\tilde{\ss}$ satisfying \eqref{EQ3DSD}. The following result is taken from \cite[Theorem~2.2]{FL15}, which shows that $\tilde{\ss}=\bar{\ss}$.

\begin{theorem}\label{THM32}
Let $(\EE,\FF)$ be in Lemma~\ref{LM31} and assume \eqref{EQ3CCI}. Then the absolutely continuous part $\bar{\ss}$ of $\ss$ is the scale function of regular Dirichlet form $(\bar{\EE},\bar{\FF})$. Particularly, $C_c^\infty(J)$ is a special standard core of $(\EE,\FF)$ if and only if $\ss$ is absolutely continuous.
\end{theorem}

\subsection{General case} Now let us consider $I=\mathbb{R}$, $m$ a fully supported Radon measure on $\mathbb{R}$ and $(\EE,\FF)$ a regular and strongly local Dirichlet form on $L^2(\mathbb{R},m)$ whose effective intervals are $\{(I_n,\ss_n):n\geq 1\}$. For each $n$, $e_n$ is a fixed point in the interior of $I_n$. The representation of $(\EE,\FF)$ is given in Theorem~\ref{THM21}. We shall point out that the assumption $I=\mathbb{R}$ is not essential, see Remark~\ref{RM39}.

We first analyse the basic assumption \eqref{EQ3CCI}. The following result follows from the fact that $C_c^\infty(\mathbb{R})\subset \FF$ if and only if {$f\in \FF_\mathrm{loc}$ for $f(x):=x$}.

\begin{lemma}\label{LM33}
Let $(\EE,\FF)$ be a regular and strongly local Dirichlet form on $L^2(\mathbb{R},m)$ with the effective intervals $\{(I_n,\ss_n):n\geq 1\}$. Then $C_c^\infty(\mathbb{R})\subset \FF$ is equivalent to the following: $dx\ll d\ss_n$ on $I_n$ for any $n\geq 1$ and for any constant $L>0$,
\begin{equation}\label{EQ3NIL}
	\sum_{n\geq 1}\int_{I_n\bigcap (-L, L)} \left(\frac{dx}{d\ss_n}\right)^2d\ss_n<\infty.
\end{equation}
\end{lemma}
\begin{proof}
This equivalent condition \eqref{EQ3NIL} is clearly a concrete expression of $f\in \FF_\mathrm{loc}$ for {$f(x):=x$}. If $C_c^\infty(\mathbb{R})\subset \FF$, then for any $L>0$ we can take a function $f^L\in C_c^\infty(\mathbb{R})$ such that $f^L(x)=f(x)$ for any $x\in (-L, L)$. This implies $f\in \FF_\mathrm{loc}$. On the contrary, assume $f\in \FF_\mathrm{loc}$ and take $\varphi \in C_c^\infty(\mathbb{R})$ with $\text{supp}[\varphi]\subset (-L,L)$. Then it follows from \eqref{EQ3NIL} that $\varphi|_{I_n}\ll \ss_n$ and
\[
\sum_{n\geq 1} \EE^{(\ss_n)}(\varphi|_{I_n}, \varphi|_{I_n})\leq \frac{\|\varphi'\|_\infty^2}{2}\sum_{n\geq 1}\int_{I_n\bigcap (-L, L)} \left(\frac{dx}{d\ss_n}\right)^2d\ss_n<\infty.
\]
Thus $\varphi\in \FF$.
That completes the proof.
\end{proof}

\begin{remark}\label{RM34}
We can deduce several corollaries from \eqref{EQ3NIL}.
\begin{itemize}
\item[(1)] For each $n\geq 1$, the inverse $\tt_n:=\ss_n^{-1}$ of $\ss_n$ is absolutely continuous.
\item[(2)] If $I_n$ is an internal interval (equivalently, a finite interval), then there exists a constant $L>0$ such that $I_n\subset (-L,L)$. Thus
\[
	\int_{\ss_n(I_n)}\tt'_n(y)^2dy<\infty.
\]
\item[(3)] If $I_n=\langle a_n,b_n\rangle$ is an infinite interval, i.e. $a_n=-\infty$ or $b_n=\infty$, then $\tt'_n$ is only a $L^2_\mathrm{loc}$-function in its domain.
\item[(4)] If in addition each $\ss_n$ is absolutely continuous, then \eqref{EQ3NIL} is the same as
\[
	\sum_{n\geq 1}\int_{I_n\bigcap (-L, L)} \frac{1}{\ss'_n(x)} dx<\infty.
\]
\end{itemize}
\end{remark}

We are now in a position to answer when $\bar{\FF}=\FF$. For this purpose, we need to give some definitions. Define two measures on $\mathbb{R}$:
\begin{equation}\label{EQ3SNS}
	\lambda_\ss:=\sum_{n\geq 1} d\ss_n|_{I_n},\quad \lambda_l:=dx|_{\left(\bigcup_{n\geq 1}I_n\right)^c}.
\end{equation}
In other words, $\lambda_\ss$ is the sum of all measures assciated to scale functions $\{\ss_n: n\geq 1\}$ on effective intervals, not necessarily Radon, and $\lambda_l$ is the restriction of Lebesgue measure to the rest part. Recall that $e_n$ is a fixed point in the interior of $I_n$ for $n\geq 1$.

\begin{definition}\label{DEF35}
For two effective intervals $I_i$ and $I_j$, we say that $I_i$ and $I_j$ are \emph{scale-connected}, or $I_i$ is scale-connected to $I_j$, with respect to $\{\ss_n: n\geq 1\}$, if
\begin{equation}\label{EQ3LSE}
	\lambda_\ss([e_i,e_j])<\infty,\quad \lambda_l([e_i,e_j])=0,
\end{equation}
where $[e_i,e_j]:=[e_j,e_i]$ if $e_j<e_i$. If $I_i$ is not scale-connected to any other effective interval, we say that it is \emph{scale-isolated}.
\end{definition}

We write $I_i\overset{\ss}{\sim} I_j$ to represent that $I_i$ is scale-connected to $I_j$. Then one may easily check that $I_i\overset{\ss}{\sim} I_i$, $I_i\overset{\ss}{\sim} I_j$ implies $I_j\overset{\ss}{\sim} I_i$, $I_i\overset{\ss}{\sim} I_j$ and $I_j\overset{\ss}{\sim} I_k$ imply $I_i\overset{\ss}{\sim} I_k$. In other words, the scale-connection is an equivalence relation on all effective intervals $\{I_n:n\geq 1\}$.

\begin{remark}\label{RM36}
Let $I_i$ and $I_j$ be two scale-connected effective intervals and $b_i< a_j$ (notice that $b_i=a_j$ is impossible). Intuitively, $I_i$ is on the left side of $I_j$.
\begin{itemize}
\item[(1)] $\lambda_\ss([e_i,e_j])<\infty$ implies there are no open endpoints between $I_i$ and $I_j$. More precisely, all the effective intervals between $I_i$ and $I_j$ are closed and $b_i$ and $a_j$ are also closed endpoint. However, $a_i$ or $b_j$ is not necessarily closed, see Example~\ref{EXA38}.
\item[(2)] The first remark and $\lambda_l([e_i,e_j])=0$ indicate the effective intervals between $I_i$ and $I_j$ have a Cantor-type structure, see (for example) Example~\ref{EXA38}~(2, 3, 4).
\item[(3)] If $I_k$ is an effective interval between $I_i$ and $I_j$, then $I_k$ is scale-connected to $I_i$ and $I_j$. Thus roughly speaking, the equivalence class of scale-connection separates all the effective intervals into several `continuous' clusters.
\end{itemize}
\end{remark}

\begin{theorem}\label{THM37}
Let $(\EE,\FF)$ be a strongly local and regular Dirichlet form on $L^2(\mathbb{R},m)$ whose effective intervals are $\{(I_n, \ss_n): n\geq 1\}$ and assume $C_c^\infty(\mathbb{R})\subset \FF$. Then $C_c^\infty(\mathbb{R})$ is a special standard core of $(\EE,\FF)$ if and only if the following conditions hold:
\begin{itemize}
\item[(1)] $\ss_n$ is absolutely continuous for any $n\geq 1$;
\item[(2)] for any $i$, $I_i$ is scale-isolated with respect to $\{\ss_n: n\geq 1\}$.
\end{itemize}
\end{theorem}
\begin{proof}
We first prove the necessity. Assume that $C_c^\infty(\mathbb{R})$ is a special standard core of $(\EE,\FF)$.

For (1), fix $n$ and recall that $I_n=\langle a_n,b_n\rangle$. The interior of $I_n$ is denoted by $\mathring{I}_n=(a_n,b_n)$. Let us consider the part Dirichlet form $(\EE_{\mathring{I}_n}, \FF_{\\mathring{I}_n})$ of $(\EE,\FF)$ on $\mathring{I}_n$. {Clearly, $(\EE_{\mathring{I}_n}, \FF_{\mathring{I}_n})= (\EE^{(\ss_n)}_{\mathring{I}_n},\FF^{(\ss_n)}_{\mathring{I}_n})$, where the latter one is the part Dirichlet form of $(\EE^{(\ss_n)},\FF^{(\ss_n)})$ on $\mathring{I}_n$.} On the other hand, since $C_c^\infty(\mathbb{R})$ is a special standard core of $(\EE,\FF)$, it follows that $C_c^\infty(\mathring{I}_n)$ is a special standard core of $(\EE_{\mathring{I}_n}, \FF_{\mathring{I}_n})$. Thus from Theorem~\ref{THM32} we can obtain that $\ss_n$ is absolutely continuous.

For (2), suppose that $I_1=\langle a_1,b_1\rangle $ and $I_2=\langle a_2,b_2\rangle$ are scale-connected. Without loss of generality, further assume $b_1\leq a_2$.  Let us consider the part Dirichlet form $(\EE_J,\FF_J)$ of $(\EE,\FF)$ on $J=(a_1,b_2)$. Clearly, $C_c^\infty(J)$ is a special standard core of $(\EE_J,\FF_J)$ and $(\EE_J,\FF_J)$ is not irreducible. Take a fixed point $e\in J$. Since $I_1\overset{\ss}{\sim} I_2$, it follows that $\lambda_\ss|_J$ is a Radon measure on $J$. Denote the induced increasing function of $\lambda_\ss|_J$ by $\ss_J$:
\[
	\ss_J(x):=\int_e^x d\lambda_\ss,\quad x\in J.
\]
We assert that both $\ss_J$ and its inverse $\tt_J:=\ss_J^{-1}$ are absolutely continuous, and $\tt'_J\in L^2_\mathrm{loc}(\ss_J(J))$. In fact, the absolute continuity of $\ss_J$ follows from (1) and $d\ss_J=\sum_{n\geq 1}d\ss_n$ and the absolute continuity of $\tt_J$ follows from Remark~\ref{RM34}~(1) and $\lambda_l(J)=0$. Take any closed interval $K\subset \ss_J(J)$, it follows from Remark~\ref{RM34}~(4) that
\[
	\int_K \tt'_J(y)^2dy=\int_{\tt_J(K)} \tt'_J\circ \ss_J(x)^2d\ss_J(x)=\int_{\tt_J(K)} \frac{1}{\ss'_J(x)}dx<\infty.
\]
From this assertion, \eqref{EQ3TIA} and Theorem~\ref{THM32}, we can conclude that $C_c^\infty(J)$ is a special standard core of $(\EE^{(\ss_J)}, \FF^{(\ss_J)})$, where $(\EE^{(\ss_J)}, \FF^{(\ss_J)})$ is given by \eqref{EQ2FSU} with $I:=J, \ss:=\ss_J$.  Now we shall verify that $$\EE^{(\ss_J)}(\varphi, \varphi)=\EE_J(\varphi, \varphi)$$ for any $\varphi\in C_c^\infty(J)$. Indeed, $\varphi$ could be treated as a function in $C_c^\infty(\mathbb{R})$ with $\varphi|_{J^c}=0$. Then we have
\[
\begin{aligned}
	\EE_J(\varphi, \varphi)&= \EE(\varphi, \varphi) \\
	&=\frac{1}{2}\sum_{n\geq 1}\int_{I_n\cap J}\left(\frac{d\varphi}{d\ss_n}\right)^2d\ss_n(x) \\
	&=\frac{1}{2}\sum_{n\geq 1}\int_{I_n\cap J}\left(\varphi'(x)\right)^2\frac{1}{\ss'_n(x)}dx  \\
	&=\frac{1}{2}\int_{J}\left(\varphi'(x)\right)^2\frac{1}{\ss'_J(x)}dx \\
	&=\EE^{(\ss_J)}(\varphi,\varphi).
\end{aligned}\]
This implies that $(\EE^{(\ss_J)}, \FF^{(\ss_J)})=(\EE_J,\FF_J)$, which leads to a contradiction because  $(\EE^{(\ss_J)}, \FF^{(\ss_J)})$ is irreducible but $(\EE_J,\FF_J)$ is not irreducible.
Therefore, the necessity is proved.

Next, we turn to prove the sufficiency. Assume that (1) and (2) hold. Denote the closure of $C_c^\infty(\mathbb{R})$ in $(\EE,\FF)$ by $\bar{\FF}$ and set $\bar{\EE}(u,v):=\EE(u,v)$ for any $u,v\in \bar{\FF}$. Then $(\bar{\EE},\bar{\FF})$ is a regular and strongly local Dirichlet form on $L^2(\mathbb{R},m)$. Applying Theorem~\ref{THM21} to $(\bar{\EE},\bar{\FF})$, denote the effective intervals of $(\bar{\EE},\bar{\FF})$ by $\{(\bar{I}_n, \bar{\ss}_n): n\geq 1\}$. Write $\bar{I}_n=\langle \bar{a}_n,\bar{b}_n\rangle$. We need to show that
\begin{equation}\label{EQ3INS}
	\{(I_n,\ss_n):n\geq 1\}= \{(\bar{I}_n,\bar{\ss}_n):n\geq 1\}.
\end{equation}

We assert the following fact: the endpoints $\bar{a}_n$ and $\bar{b}_n$ are not the interior points of some $I_k$. Indeed, if $\bar{a}_n\in \mathring{I}_k$, consider the part Dirichlet forms of $(\EE,\FF)$ and $(\bar{\EE},\bar{\FF})$ on $J:=\mathring{I}_k$. Denote the part Dirichlet forms by $(\EE_J,\FF_J)$ and $(\bar{\EE}_J, \bar{\FF}_J)$ respectively. Clearly, $(\EE_J,\FF_J)$ is irreducible but $(\bar{\EE}_J, \bar{\FF}_J)$ is not irreducible since $\bar{I}_n\cap J$ is a non-trivial invariant set of $(\bar{\EE}_J, \bar{\FF}_J)$. This is in contradiction with \cite[Proposition~2.3~(3)]{LY16} since $(\bar{\EE}_J, \bar{\FF}_J)$ is a regular Dirichlet subspace of $(\EE_J,\FF_J)$.

 Given any $\bar{I}_n$, we shall prove that there exists some $k$ such that $\bar{I}_n=I_k$ and $\bar{\ss}_n=\ss_k$. In other words,
\begin{equation}\label{EQ3INS-1}
	\{(I_n,\ss_n):n\geq 1\}\supset \{(\bar{I}_n,\bar{\ss}_n):n\geq 1\}.
\end{equation}
In fact, it follows from the above fact that for any $j$, either $\mathring{I}_j\cap \bar{I}_n=\emptyset$ or $\mathring{I}_j\subset \bar{I}_n$. Denote all the intervals with interior included in $\bar{I}_n$ by
\begin{equation}\label{EQ3IKJ}
	\{{I}_{k_j}: j\geq 1 \}.
\end{equation}
Note that they are disjoint with each other. We assert that any two $I_{k_1}$ and $I_{k_2}$ are scale-connected, thus \eqref{EQ3IKJ} has at most one element by the condition (2). To this end, for any $j$, consider the part Dirichlet forms of $(\EE,\FF)$ and $(\bar{\EE}, \bar{\FF})$ on $J:=\mathring{I}_{k_j}$. We still denote them by $(\EE_J,\FF_J)$ and $(\bar{\EE}_J, \bar{\FF}_J)$. Clearly, we have
\[
	(\EE_J,\FF_J)=(\EE^{(\ss_{k_j})}_J, \FF^{(\ss_{k_j})}_J), \quad (\bar{\EE}_J, \bar{\FF}_J)= (\EE^{(\bar{\ss}_n)}_J, \FF^{(\bar{\ss}_n)}_J)
\]
on $L^2(J, m|_J)$. Since $C_c^\infty(J)$ is a special standard core of $(\bar{\EE}_J, \bar{\FF}_J)$, it follows from Theorem~\ref{THM32} that $\bar{\ss}_n$ is absolutely continuous on $J$. Moreover, the inclusion $C_c^\infty(J)\subset \bar{\FF}_J\subset \FF_J$ indicates $\bar{\ss}'_n>0, \ss'_{k_j}>0$ a.e. on $J$  by \eqref{EQ3TIA}. Notice the $\ss_{k_j}$ is absolutely continuous by (1). On the other hand, $(\bar{\EE}_J, \bar{\FF}_J)$ is a regular Dirichlet subspace of $(\EE_J,\FF_J)$. This implies by Lemma~\ref{LM31}
\[
	\frac{d\bar{\ss}_n}{d\ss_{k_j}}=\frac{\bar{\ss}'_n}{\ss'_{k_j}}=0\text{ or }1,\quad \text{a.e. on }J.
\]
Since $\bar{\ss}'_n/\ss'_{k_j}>0$ a.e. on $J$, we can conclude that $d\bar{\ss}_n=d\ss_{k_j}$ on $J$. Particularly, $\lambda_\ss=d\bar{\ss}_n$ on $\bar{I}_n$ and thus
$$\lambda_{\ss}([e_{k_1},e_{k_2}])=\lambda_{\bar{\ss}}([e_{k_1},e_{k_2}])<\infty.$$
Furthermore, for any $\varphi \in C_c^\infty((\bar{a}_n,\bar{b}_n))$, it follows from $\EE(\varphi, \varphi)=\bar{\EE}(\varphi, \varphi)$ that
\[
\int_{\bar{a}_n}^{\bar{b}_n}\left(\frac{d\varphi}{d\bar{\ss}_n} \right)^2d\bar{\ss}_n=\sum_{j}\int_{I_{k_j}}\left(\frac{d\varphi}{d\ss_{k_j}} \right)^2d\ss_{k_j}=\sum_{j}\int_{I_{k_j}} \left(\frac{d\varphi}{d\bar{\ss}_n} \right)^2d\bar{\ss}_n.
\]
Therefore, we have $$\lambda_l(\bar{I}_n)=|\bar{I}_n\setminus (\bigcup_jI_{k_j})|=0$$ and it follow that $I_{k_1}$ is scale-connected to $I_{k_2}$. In addition, the fact that $\lambda_l(\bar{I}_n)=0$ also implies \eqref{EQ3IKJ} has at least one element. Hence combining together, \eqref{EQ3IKJ} has exactly one element, which is denoted by $\mathring{I}_k$, and $\mathring{I}_k=(\bar{a}_n,\bar{b}_n)$ and $d\ss_k=d\bar{\ss}_n$. Notice that $\ss_k$ and $\bar{\ss}_n$ satisfy \textbf{(A)} and \textbf{(B)} (see \eqref{EQ2SJS}), and the fixed point $e_k$ in $\mathring{I}_k=(\bar{a}_n,\bar{b}_n)$ for $\ss_k$ and $\bar{\ss}_n$ (i.e. $\ss_k(e_k)=\bar{\ss}_n(e_k)=0$) could be taken as the same one. We then conclude that $I_k=\bar{I}_n$ and $\ss_k=\bar{\ss}_n$.

Finally, suppose that some $(I_k, \ss_k)$ is not contained in the right side of \eqref{EQ3INS}. Let us consider the the part Dirichlet forms $(\EE_J,\FF_J)$ and $(\bar{\EE}_J, \bar{\FF}_J)$ of $(\EE,\FF)$ and $(\bar{\EE},\bar{\FF})$ on $J:=\mathring{I}_k$. Note that $C_c^\infty(J)\subset \bar{\FF}_J\subset \FF_J$ and for any $\varphi\in C_c^\infty(J)$, $\EE_J(\varphi, \varphi)=\bar{\EE}(\varphi, \varphi)$. However, $\bar{\EE}_J(\varphi, \varphi)=0$ but
\[
	\EE_J(\varphi, \varphi)= \frac{1}{2}\int_J\left(\frac{d\varphi}{d\ss_k} \right)^2d\ss_k,
\]
which leads to a contradiction. That completes the proof.
\end{proof}

\begin{example}\label{EXA38}
As characterized in Theorem~\ref{THM21}, the Dirichlet form $(\EE,\FF)$ on $L^2(\mathbb{R},m)$ is constituted by a sequence (possibly, none, finite, or infinite)  of effective intervals with scale functions $\{(I_n,\ss_n):n\geq 1\}$.
\begin{itemize}
\item[(1)] The first example is very simple. Let \begin{align*}&I_1=(-\infty, 0],\ \ss_1(x)=x;\\
&I_2=(0, \infty),\ \ss_2(x)=-\exp\{1/x\}.\end{align*} Then $\{(I_i,\ss_i): i=1,2\}$ satisfies all the conditions including (1) and (2) in Theorem~\ref{THM37}. Particularly, $I_1$ and $I_2$ are scale-isolated with respect to $\{\ss_1, \ss_2\}$. As a result, $C_c^\infty(\mathbb{R})$ is a special standard core of the associated Dirichlet form.
\item[(2)] The second example is taken from \cite[Example~3.20]{LY16}. Let $K$ be the standard Cantor set in $[0,1]$. Set $U:=K^c$ and write $U$ as a union of disjoint open intervals:
\[
U=\bigcup_{n\geq 1}(a_n,b_n),
\]
where $(a_1,b_1)=(-\infty,0), (a_2,b_2)=(1,\infty)$. Let $I_1:=(-\infty, 0], I_2:=[1,\infty)$, $I_n:=[a_{n},b_{n}]$ for any $n\geq 3$. For each $n$, let $\ss_n(x)=x$ on $I_n$. Then the associated Dirichlet form $(\EE,\FF)$ is a so-called regular Dirichlet extension of one-dimensional Brownian motion (Cf. \cite{LY16}) and the associated diffusion process $X$ is a reflected Brownian motion on each interval $I_n$. One may easily check that $I_1$ is scale-connected to $I_2$ and thus for any $i,j$, $I_i$ is scale-connected to $I_j$. Hence $C_c^\infty(\mathbb{R})$ is not a special standard core of $(\EE,\FF)$ by Theorem~\ref{THM37}.  Indeed, as indicated in \cite{LY16}, the closure of $C_c^\infty(\mathbb{R})$ in $(\EE,\FF)$ corresponds to the one-dimensional Brownian motion.
\item[(3)] We consider the intervals $\{I_n:n\geq 1\}$ in the second example but redefine the scale functions as follows: $\ss_1(x)=x, \ss_2(x)=x$ and
\[
	\ss_n(x):=\frac{x}{|b_n-a_n|},\quad x\in I_n, n\geq 3.
\]
Then we have $d\ss_n(I_n)=1$ for any $n\geq 3$. Clearly, $(\EE,\FF)$ satisfies \eqref{EQ3NIL} and thus $C_c^\infty(\mathbb{R})\subset \FF$. Notice that for any $i\neq j$, there are infinite effective intervals between $I_i$ and $I_j$. This implies that each $I_i$ is scale-isolated since $\lambda_\ss([e_i, e_j])=\infty$. Therefore from Theorem~\ref{THM37} it follows that $C_c^\infty(\mathbb{R})$ is a special standard core of $(\EE,\FF)$.
\item[(4)] Finally, we consider an example similar to the second one but replace $K$ by a generalized Cantor set. In other words, $K$ has a Cantor-type structure but the Lebesgue measure of $K$ is positive. We refer to \cite[\S1.5]{F99} for the details to construct the generalized Cantor set. The effective intervals with the scale functions are given by the same method as the second example. Then \eqref{EQ3NIL} holds and $C_c^\infty(\mathbb{R})\subset \FF$. For any $i\neq j$, clearly $\lambda_\ss([e_i,e_j])<\infty$ since $d\ss_n$ is nothing but the Lebesgue measure. However, one may easily show that $\lambda_l([e_i,e_j])>0$ from the structure of generalized Cantor set. This indicates that each $I_i$ is scale-isolated and therefore, $C_c^\infty(\mathbb{R})$ is a special standard core of $(\EE,\FF)$.
\end{itemize}
\end{example}

We end this section with a remark to explain that the setting $I=\mathbb{R}$ is not essential as we pointed out in the beginning of this section.

\begin{remark}\label{RM39}
When $I$ is an open interval, all the considerations above are robust and only very few details need to be modified. For example, $(-L, L)$ in \eqref{EQ3NIL} could be replaced by a compact sub-interval of $I$.

If $I$ is not open, we can apply the trick in \S\ref{SEC231} to find an essentially equivalent Dirichlet form $(\EE^*,\FF^*)$ on a larger open interval and transfer the questions to the new Dirichlet form. For simplicity, let us consider the following example: $I=[0,\infty)$ and $(\EE,\FF)$ is a regular and strongly local Dirichlet form on $L^2(I,m)$. Let $(\EE^*,\FF^*)$ on $L^2(\mathbb{R},m^*)$ be given by \eqref{EQ2FUL2}, which is an essentially equivalent Dirichlet form of $(\EE,\FF)$. Clearly, $C_c^\infty(\mathbb{R})\subset \FF^*$ is equivalent to
\[
	C_c^\infty(I)=\{\varphi|_I:\varphi \in C_c^\infty(\mathbb{R})\}\subset \FF.
\]
We assert that $C_c^\infty(\mathbb{R})$ is a special standard core of $(\EE^*,\FF^*)$ if and only if $C_c^\infty(I)$ is a special standard core of $(\EE,\FF)$. The necessity is clear. For the sufficiency, assume that $C_c^\infty(I)$ is a special standard core of $(\EE,\FF)$, we need to prove $C_c^\infty(\mathbb{R})$ is a core of $(\EE^*,\FF^*)$. Consider the part Dirichlet form $(\EE_0, \FF_0)$ of $(\EE,\FF)$ on $\mathring{I}=(0, \infty)$. Note that $C_c^\infty((0,\infty))$ is a special standard core of $(\EE_0,\FF_0)$. Mimicking Theorem~\ref{THM37}, we can deduce that the effective intervals of $(\EE_0, \FF_0)$ satisfy the conditions (1) and (2) of Theorem~\ref{THM37}. Notice that the effective intervals are almost the same as those of $(\EE,\FF)$ (except for the case $[0, b\rangle$ is an effective interval of $(\EE,\FF)$ but the corresponding one of $(\EE_0,\FF_0)$ is $(0, b\rangle$). One easily sees that $(\EE,\FF)$ (hence $(\EE^*,\FF^*)$) also satisfies the conditions (1) and (2) of Theorem~\ref{THM37}. Applying Theorem~\ref{THM37} to $(\EE^*,\FF^*)$ again, we conclude that $C_c^\infty(\mathbb{R})$ is a special standard core of $(\EE^*,\FF^*)$.
\end{remark}

\subsection{Recovering Hamza's theorem}

A class of strongly local and regular Dirichlet forms on $L^2(\mathbb{R}, m)$ was introduced by Hamza \cite{H75}, which starts from an energy form defined on $C_c^\infty(\mathbb{R})$. For convenience, we introduce a stronger version which is stated in \cite{FOT11}. Let $\nu$ be a positive Radon measure on $\mathbb{R}$ and define the following quadratic form:
\begin{equation}\label{EQ3DEC}
\begin{aligned}
	\mathcal{D}(\mathcal{E})&:=C_c^\infty(\mathbb{R}),\\
	\EE(u,v)&:=\frac{1}{2}\int_\mathbb{R} u'(x)v'(x)\nu(dx),\quad u,v\in \mathcal{D}(\EE).
\end{aligned}\end{equation}
The aim is to find some conditions on $\nu$, under which $(\EE,\mathcal{D}(\EE))$ is closable on $L^2(\mathbb{R},m)$. Given a non-negative function $a(x)$ on $\mathbb{R}$, we say $x\in \mathbb{R}$ is a regular point of $a$ if for some $\epsilon>0$,
\begin{equation}\label{EQ3XEX}
\int_{x-\epsilon}^{x+\epsilon}\frac{1}{a(\xi)}d\xi<\infty.
\end{equation}
The set of all regular points of $a$ are denoted by $R(a)$ and $S(a):=\mathbb{R}\setminus R(a)$ is called the singular set of $a$. The following theorem is taken from \cite[Theorem~3.1.6]{FOT11} and we give an alternative proof of the necessity, and the sufficiency is left for further discussions afterwards.

\begin{theorem}\label{THM410}
The form \eqref{EQ3DEC} is closable on $L^2(\mathbb{R},m)$ if and only if the following conditions are satisfied:
\begin{itemize}
\item[(1)] $\nu$ is absolutely continuous,
\item[(2)] the density function $a$ ($\nu(dx)=a(x)dx$) vanished a.e. on its singular set $S(a)$.
\end{itemize}
\end{theorem}
\begin{proof}[Proof of necessity]
Denote the closure of $\mathcal{D}(\EE)$ relative to $\EE_1$ by $\FF$. Then $(\EE,\FF)$ is a regular and strongly local Dirichlet form on $L^2(\mathbb{R},m)$ and $C_c^\infty(\mathbb{R})$ is a special standard core of $(\EE,\FF)$. Applying Theorem~\ref{THM21} to $(\EE,\FF)$, denote its effective intervals by $\{(I_n, \ss_n): n\geq 1\}$. It follows from Lemma~\ref{LM33} and Theorem~\ref{THM37} that $\ss_n$ is absolutely continuous and $\ss'_n>0$ a.e. on $I_n$. Moreover, for any $\varphi \in C_c^\infty(\mathbb{R})$, we have
\begin{equation}\label{EQ3EVV}
	\EE(\varphi, \varphi)= \frac{1}{2}\sum_{n\geq 1}\int_{I_n}\left(\frac{d\varphi}{d\ss_n}\right)^2d\ss_n=\frac{1}{2}\sum_{n\geq 1}\int_{I_n}\left(\varphi'(x)\right)^2\frac{1}{\ss'_n(x)}dx.
\end{equation}
Define a measure $\mu$ as follows:
\[
\mu|_{\left\{\mathbb{R}\setminus \bigcup_{n\geq 1}I_n\right\}}=0,\quad  \mu|_{I_n}:= \frac{1}{\ss'_n(x)}dx,\quad n\geq 1.	
\]
Then Lemma~\ref{LM33} and Remark~\ref{RM34}~(4) imply that $\mu$ is a Radon measure on $\mathbb{R}$ and from \eqref{EQ3DEC} and \eqref{EQ3EVV} we can deduce that $\mu=\nu$. Particularly, $\nu$ is absolutely continuous, i.e. (1) holds. From the fact each $I_i$ is scale-isolated (see Theorem~\ref{THM37}), we conclude that the regular set of the density function $a$ of $\nu$ is
\[
	R(a)=\bigcup_{n\geq 1} \mathring{I}_n,
\]
where $\mathring{I}_n$ is the interior of $I_n$. Clearly, the density function $a$ vanishes a.e. on $S(a)$. That completes the proof.
\end{proof}

Assume the quadratic form \eqref{EQ3DEC} satisfies the two conditions in Theorem~\ref{THM410} and denote the closure of $\mathcal{D}(\EE)$ relative to $\EE_1$ by $\FF$. Clearly, $(\EE,\FF)$ is a strongly local and regular Dirichlet form on $L^2(\mathbb{R})$, and $C_c^\infty(\mathbb{R})$ is a special standard core of $(\EE,\FF)$. The main purpose of rest part is to identify the effective intervals of $(\EE,\FF)$, so that we can obtain an explicit structure of the associated diffusion process.

We first note that the condition that $\nu$ is Radon plays a similar role to \eqref{EQ3CCI} since it guarantees that the integration in \eqref{EQ3DEC} is finite. On the other hand, from \eqref{EQ3XEX} we can easily deduce that the regular set $R(a)$ is open. Thus we can write $R(a)$ as a union of disjoint open intervals:
\begin{equation}\label{EQ3RAN}
	R(a)=\bigcup_{n\geq 1} (a_n,b_n).
\end{equation}
For each $n$ take a fixed point $e_n\in (a_n,b_n)$ and define
\begin{equation}\label{EQ3SNX}
	\ss_n(x):=\int_{e_n}^x \frac{1}{a(\xi)}d\xi,\quad x\in (a_n,b_n).
\end{equation}
We assert $\ss_n$ is an absolutely continuous scale function on $(a_n,b_n)$, i.e. $\ss_n$ is absolutely continuous and strictly increasing. In fact, one may easily verify that $|\ss_n(x)|<\infty$ by \eqref{EQ3XEX} and Heine-Borel theorem. It follows that $1/a\in L^1_\mathrm{loc}((a_n,b_n))$ and $\ss_n$ is absolutely continuous. Clearly, $\ss_n$ is increasing, and it is actually strictly increasing. In fact suppose that $\ss_n(x)=\ss_n(y)$ for $a_n<x<y<b_n$. Then $a(\xi)=\infty$ for any $\xi\in (x,y)$, which contradicts that $\nu(d\xi)=a(\xi)d\xi$ is Radon.

Naturally, we can extend the definition of $\ss_n$ to
\[
	\ss_n(a_n):=\lim_{x\downarrow a_n}\ss_n(x)(\geq -\infty),\quad \ss_n(b_n):=\lim_{x\uparrow b_n}\ss_n(x)(\leq \infty).
\]
Set
\begin{equation}\label{EQ3INA}
	I_n:=\langle a_n, b_n\rangle,
\end{equation}
where $a_n\in I_n$ (resp. $b_n\in I_n$) if and only if $a_n>-\infty, \ss_n(a_n)>-\infty$ (resp. $b_n<\infty, \ss_n(b_n)<\infty$). Then $\{I_n: n\geq 1\}$ is mutually disjoint.  Indeed, suppose that $I_i=\langle a_i, b_i], I_j=[a_j, b_j\rangle$ and $b_i=a_j$. Then $\ss_i(b_i)<\infty$ and $\ss_j(a_j)>-\infty$ imply the common endpoint $b_i=a_j$ is a regular point of $a$. However, \eqref{EQ3RAN} tells us $b_i=a_j\notin R(a)$, which leads to a contradiction.

\begin{corollary}\label{COR411}
Let $(\EE,\FF)$ be the closure of \eqref{EQ3DEC} on $L^2(\mathbb{R},m)$ under the conditions of Theorem~\ref{THM410}. The interval $I_n$ and the scale function $\ss_n$ on $I_n$ are defined by \eqref{EQ3INA} and \eqref{EQ3SNX} respectively for each $n\geq 1$. Then $(\EE,\FF)$ can be expressed as \eqref{EQ2FUL} with the effective intervals $\{(I_n, \ss_n): n\geq 1\}$.
\end{corollary}

\begin{proof}
Let $(\bar{\EE},\bar{\FF})$ be the Dirichlet form given by \eqref{EQ2FUL} with effective intervals $\{(I_n, \ss_n): n\geq 1\}$. We need to show $(\EE,\FF)=(\bar{\EE}, \bar{\FF})$.

We first assert that $C_c^\infty(\mathbb{R})\subset \bar{\FF}$. In fact, since $\ss'_n=1/a>0$ a.e. on $I_n$, it follows that $\tt_n:=\ss_n^{-1}$ is absolutely continuous. By Lemma~\ref{LM33}, Remark~\ref{RM34}~(4) and
\[
	\sum_{n\geq 1} \int_{I_n\bigcap (-L, L)}\frac{1}{\ss'_n(x)}dx=\sum_{n\geq 1} \int_{I_n\bigcap (-L, L)}a(x)dx=\nu((-L,L))<\infty, \quad \forall L>0,
\]
it can be seen that $C_c^\infty(\mathbb{R})\subset \bar{\FF}$.
Moreover, for any $\varphi\in C_c^\infty(\mathbb{R})$, from Remark~\ref{RM34}~(4) and the second condition of Theorem~\ref{THM410}, we can deduce that
\[
\begin{aligned}
	\bar{\EE}(\varphi, \varphi) &= \frac{1}{2}\sum_{n\geq 1}\int_{I_n}\left(\frac{d\varphi}{d\ss_n}\right)^2d\ss_n \\
	&=\frac{1}{2}\sum_{n\geq 1}\int_{I_n}\varphi'(x)^2\frac{1}{\ss'_n(x)}dx \\
	&=\frac{1}{2}\sum_{n\geq 1}\int_{I_n}\varphi'(x)^2a(x)dx \\
	&=\frac{1}{2}\int_\mathbb{R}\varphi'(x)^2\nu(dx) \\
	&=\EE(\varphi, \varphi).
\end{aligned}\]
Thus we need only to show that $C_c^\infty(\mathbb{R})$ is a special standard core of $(\bar{\EE},\bar{\FF})$. By Theorem~\ref{THM37}, it suffices to prove that each $I_i$ is scale-isolated with respect to $\{\ss_n: n\geq 1\}$. Suppose that $I_i$ is scale-connected to $I_j$ for some $j\neq i$. On one hand, there exists some point $x$ between $I_i$ and $I_j$, which is not a regular point of $a$ by \eqref{EQ3RAN}. On the other hand, the measure $\lambda_\ss$ defined by \eqref{EQ3SNS}  $$\lambda_{\ss}=(1/a(\xi))\cdot d\xi|_{\cup_{n\geq 1}I_n}.$$ Thus from $\lambda_l([e_i,e_j])=0$, we can obtain that for some $\epsilon>0$ such that $(x-\epsilon,x+\epsilon)\subset [e_i, e_j]$,
\[
	\int_{x-\epsilon}^{x+\epsilon}\frac{1}{a(\xi)}d\xi=\int_{x-\epsilon}^{x+\epsilon} d\lambda_\ss\leq \lambda_\ss([e_i,e_j])<\infty.
\]
In other words, $x\in R(a)$, which leads to a contradiction. That completes the proof.
\end{proof}

\subsection{Closure of $C^{\infty}_c$}

Let $(\EE,\FF)$ be a strongly local and regular Dirichlet form on $L^2(\mathbb{R},m)$ with the effective intervals $\{(I_n, \ss_n): n\geq 1\}$ and assume that $$C_c^\infty(\mathbb{R})\subset \FF.$$  We still denote the Dirichlet form associated with the closure of $C_c^\infty(\mathbb{R})$ in $(\EE,\FF)$ by $(\bar{\EE},\bar{\FF})$. This section is devoted to explore the structure of $(\bar{\EE}, \bar{\FF})$, which is not necessarily accordant with $(\EE,\FF)$. Since $(\bar{\EE},\bar{\FF})$ is also a regular and strongly local Dirichlet form on $L^2(\mathbb{R},m)$, it follows from Theorem~\ref{THM21} that $(\bar{\EE},\bar{\FF})$ admits the expression of \eqref{EQ2FUL} with another sequence of effective intervals with scale functions, which is our essential target.

Before presenting the main theorem, we need to do some preparation. As the irreducible case in \S\ref{SEC32}, the scale function $\ss_n$ is not necessarily absolutely continuous. However we may write its Lebesgue decomposition as follows:
\[
	d\ss_n=g_n(x)dx+\kappa_n, \quad \kappa_n \perp dx.
\]
Mimicking \S\ref{SEC32}, we can define the absolutely continuous part of $\ss_n$ by
\[
	\bar{\ss}_n(x):=\int_{e_n}^xg_n(y)dy,\quad x\in I_n.
\]
We know that $\bar{\ss}_n$ is an absolutely continuous scale function on $I_n$. One may guess the effective intervals of $(\bar{\EE}, \bar{\FF})$ are $\{(I_n, \bar{\ss}_n): n\geq 1\}$. But this is not correct. A counterexample is Example~\ref{EXA38}~(2), in which $\ss_n$ is absolutely continuous but $(\bar{\EE},\bar{\FF})$ corresponds to the one-dimensional Brownian motion. Another example below shows that the scale function $\bar{\ss}_n$ may not be adapted to $I_n$, i.e. $\bar{\ss}_n\notin \SS_\infty(I_n)$, where $\SS_\infty(I_n)$ is given by \eqref{EQ2SJS}.

\begin{example}
Let $J:=(0,1]$. Further let $\mathfrak{c}$ be the standard Cantor function on $[0, 1]$, particularly $\mathfrak{c}(0)=0, \mathfrak{c}(1)=1$. Then $\mathfrak{c}$ induces a singular (i.e. $d\mathfrak{c}\perp dx$) probability measure $d\mathfrak{c}$ on $(0,1)$. For each integer $p\geq 1$, consider the map
\[
f_p: (0,1)\rightarrow \left(\frac{1}{2^p}, \frac{1}{2^{p-1}}\right),\quad x\mapsto \frac{1+x}{2^p}
\]
and denote the image measure of $d\mathfrak{c}$ relative to $f_p$ by $d\mathfrak{c}_p:=d\mathfrak{c}\circ f_p^{-1}$. Then $d\mathfrak{c}_p$ is a probability measure on $\left(\frac{1}{2^p}, \frac{1}{2^{p-1}}\right)$. Clearly, $\sum_{p\geq 1}d\mathfrak{c}_p+dx$ defines a Radon measure on $J=(0,1]$. Define $\ss$ to be the increasing function associated to this measure, in other words, $\ss(1)=1$ and
\begin{equation}\label{EQ3SNX2}
	 d\ss=dx+ \sum_{p\geq 1}d\mathfrak{c}_p.
\end{equation}
Note that $\ss(0):=\lim_{x\downarrow 0}\ss(x)=-\infty$. Thus $\ss\in \SS_\infty(J)$. On the other hand, since $d\mathfrak{c}_p\perp dx$, it follows that \eqref{EQ3SNX2} is the Lebesgue decomposition of $d\ss$. This implies that $d\bar{\ss}(x)=dx$ and $\bar{\ss}\notin \SS_\infty(J)$.
\end{example}

From these examples, we can see that both intervals and scale functions will change. Recall the definition of scale-connection for the effective intervals in Definition~\ref{DEF35}. Similarly we use the new family of scale functions $\{\bar{\ss}_n\}$ to define scale-connection. Precisely set
\[
	\lambda_{\bar{\ss}}:=\sum_{n\geq 1} d\bar{\ss}_n|_{I_n}.
\]
Two effective intervals $I_i$ and $I_j$ are scale-connected with respect to $\{\bar{\ss}_n: n\geq 1\}$ if
\[
\lambda_{\bar{\ss}}([e_i, e_j])<\infty,\quad  \lambda_l([e_i,e_j])=0.
\]
Note that $\lambda_{\bar{\ss}}\leq \lambda_\ss$. It follows that the scale-connection with respect to $\{\ss_n: n\geq 1\}$ implies the scale-connection with respect to $\{\bar{\ss}_n: n\geq 1\}$ but not vice versa. We write $I_i\overset{\bar{\ss}}{\sim} I_j$ to represent that $I_i$ is scale-connected to $I_j$ with respect to $\{\bar{\ss}_n:n\geq 1\}$.  Clearly, `$\overset{\bar{\ss}}{\sim}$' is also an equivalence relation on all effective intervals $\{I_n:n\geq 1\}$. Denote all the equivalence classes induced by $\overset{\bar{\ss}}{\sim}$ by
\[
	\bar{\mathfrak{I}}_k=\{I_{k_j}: j\geq 1\}, \quad k\geq 1,
\]
where $k$ is the index of the equivalence class and $\{k_j: j\geq 1\}$ are the indexes of the effective intervals in $k$-th equivalence class. Note that both numbers of equivalence classes and effective intervals in each equivalence class may be finite or infinite. As described in Remark~\ref{RM36}~(3), $\bar{\mathfrak{I}}_k$ also looks like a `continuous' cluster.

Fix $k$, and define
\[
	\bar{\mathfrak{a}}_k:= \inf\{x\in I_{k_j}: j\geq 1\},\quad \bar{\mathfrak{b}}_k:=\sup\{x\in I_{k_j}: j\geq 1\}.
\]
In other words, $[\bar{\mathfrak{a}}_k, \bar{\mathfrak{b}}_k]$ is the smallest closed interval, which contains all the intervals in the equivalence class $\bar{\mathfrak{I}}_k$ (the notation $[\bar{\mathfrak{a}}_k, \bar{\mathfrak{b}}_k]$ is a little abused since $\bar{\mathfrak{a}}_k$ or $\bar{\mathfrak{b}}_k$ may be infinity). One may easily check that $\lambda_{\bar{\ss}}|_{(\bar{\mathfrak{a}}_k, \bar{\mathfrak{b}}_k)}$ is a Radon measure on $(\bar{\mathfrak{a}}_k, \bar{\mathfrak{b}}_k)$. Then we can take a fixed point $\bar{\mathfrak{e}}_k\in (\bar{\mathfrak{a}}_k, \bar{\mathfrak{b}}_k)$ and define
\[
\bar{\mathfrak{s}}_k(x):=\int_{\bar{\mathfrak{e}}_k}^x d\lambda_{\bar{\ss}},\quad x\in [\bar{\mathfrak{a}}_k, \bar{\mathfrak{b}}_k].
\]
Clearly, $\bar{\mathfrak{s}}_k$ is a scale function on $[\bar{\mathfrak{a}}_k, \bar{\mathfrak{b}}_k]$ and $\bar{\mathfrak{s}}_k(\bar{\mathfrak{a}}_k)\geq -\infty, \bar{\mathfrak{s}}_k(\bar{\mathfrak{b}}_k)\leq \infty$.
Further define
\[
	\bar{\mathtt{I}}_k:= \langle \bar{\mathfrak{a}}_k, \bar{\mathfrak{b}}_k \rangle,
\]
where $\bar{\mathfrak{a}}_k \in \bar{\mathtt{I}}_k$ (resp. $\bar{\mathfrak{b}}_k \in \bar{\mathtt{I}}_k$) if and only if $\bar{\mathfrak{a}}_k>-\infty, \bar{\mathfrak{s}}_k(\bar{\mathfrak{a}}_k)>-\infty$ (resp. $\bar{\mathfrak{b}}_k<\infty, \bar{\mathfrak{s}}_k(\bar{\mathfrak{b}}_k)<\infty$). We need to check that these intervals $\{\bar{\mathtt{I}}_k: k\geq 1\}$ are mutually disjoint. Indeed, if $\bar{\mathfrak{b}}_1=\bar{\mathfrak{a}}_2\in \bar{\mathtt{I}}_1\cap \bar{\mathtt{I}}_2$, then any interval in $\bar{\mathfrak{I}}_1$ is scale-connected to any interval in $\bar{\mathfrak{I}}_2$ with respect to $\{\bar{\ss}_n: n\geq 1\}$. In other words, $\bar{\mathfrak{I}}_1=\bar{\mathfrak{I}}_2$, which contradicts the definition of equivalence classes.

Intuitively, as described above we merge all the intervals in $\bar{\mathfrak{I}}_k$ into a new one $\bar{\mathtt{I}}_k$ (notice that $\lambda_l(\bar{\mathtt{I}}_k)=0$) and obtain a new scale function $\bar{\mathfrak{s}}_k$ by the sum of $\{\bar{\ss}_{k_j}: j\geq 1\}$. Now we give the following definition.

\begin{definition}
The pair $(\bar{\mathtt{I}}_k, \bar{\mathfrak{s}}_k)$ is called the $C^\infty$-merging of the equivalence class $\bar{\mathfrak{I}}_k$ for any $k\geq 1$.
\end{definition}

In Example~\ref{EXA38}~(2), $\bar{\ss}_n=\ss_n$ and all the effective intervals are scale-connected to each other. Thus there is only one equivalence class $\bar{\mathfrak{I}}_1$ and we have $\bar{\mathfrak{a}}_1=-\infty, \bar{\mathfrak{b}}_1=\infty$. Take the fixed point $\bar{\mathfrak{e}}_1=0$ and we can obtain $\bar{\mathfrak{s}}_1(x)=x$ and $\bar{\mathtt{I}}_1=\mathbb{R}$. Recall the $(\bar{\EE},\bar{\FF})=(\frac{1}{2}\mathbf{D}, H^1(\mathbb{R}))$, which is the associated Dirichlet form of one-dimensional Brownian motion. Thus the $C^\infty$-merging $(\bar{\mathtt{I}}_1, \bar{\mathfrak{s}}_1)$ of  $\bar{\mathfrak{I}}_1$ is exactly the effective interval of $(\bar{\EE},\bar{\FF})$. The following theorem tells us  for general cases all $C^\infty$-merging constitute actually the effective intervals of $(\bar{\EE},\bar{\FF})$.

\begin{theorem}\label{THM414}
Let $(\EE,\FF)$ be a regular and strongly local  Dirichlet form with the effective intervals $\{(I_n, \ss_n): n\geq 1\}$ on $L^2(\mathbb{R},m)$ and assume $C_c^\infty(\mathbb{R})\subset \FF$. Further let $\bar{\ss}_n$ be the absolutely continuous part of $\ss_n$ and $\{\bar{\mathfrak{I}}_k:k\geq 1\}$ be the equivalence classes induced by the scale-connection with respect to $\{\bar{\ss}_n: n\geq 1\}$. Then the closure $(\bar{\EE},\bar{\FF})$ of $C_c^\infty(\mathbb{R})$ in $(\EE,\FF)$ can be expressed as \eqref{EQ2FUL} with the effective intervals $\{(\bar{\mathtt{I}}_k, \bar{\mathfrak{s}}_k): k\geq 1\}$, where $(\bar{\mathtt{I}}_k, \bar{\mathfrak{s}}_k)$ is the $C^\infty$-merging of $\bar{\mathfrak{I}}_k$ for all $k$.
\end{theorem}
 \begin{proof}
 Let $(\tilde{\EE},\tilde{\FF})$ be the Dirichlet form \eqref{EQ2FUL} with the effective intervals $\{(\bar{\mathtt{I}}_k, \bar{\mathfrak{s}}_k): k\geq 1\}$. It suffices to show $(\tilde{\EE},\tilde{\FF})=(\bar{\EE},\bar{\FF})$.

 We first assert $C_c^\infty(\mathbb{R})\subset \tilde{\FF}$. In fact, since $g_n(x)>0$ a.e. for any $n$ and $\lambda_l(\bar{\mathtt{I}}_k)=0$, it follows that $\bar{\mathfrak{s}}_k$ is absolutely continuous on $\bar{\mathtt{I}}_k$ and $\bar{\mathfrak{s}}'_k>0$ a.e. on $\bar{\mathtt{I}}_k$. Thus the inverse function $\bar{\mathfrak{t}}_k:=\bar{\mathfrak{s}}_k^{-1}$ is also absolutely continuous. Furthermore, for any $L>0$,
 \[
 	\sum_{k} \int_{\bar{\mathtt{I}}_k\bigcap (-L, L)} \frac{1}{\bar{\mathfrak{s}}'_k(x)}dx = \sum_{n} \int_{I_n \bigcap (-L, L)} \frac{1}{g_n(x)}dx.
 \]
 Note that
 \[
 	\frac{1}{g_n(x)}dx=\left(\frac{dx}{d\ss_n} \right)^2d\ss_n.
 \]
 Hence we have
 \[
 \sum_{k} \int_{\bar{\mathtt{I}}_k\bigcap (-L, L)} \frac{1}{\bar{\mathfrak{s}}'_k(x)}dx = \sum_{n} \int_{I_n \bigcap (-L, L)} \left(\frac{dx}{d\ss_n} \right)^2d\ss_n<\infty
 \]
 by Lemma~\ref{LM33} and $C_c^\infty(\mathbb{R})\subset \FF$. Applying Lemma~\ref{LM33} and Remark~\ref{RM34}~(4) to $\{(\bar{\mathtt{I}}_k, \bar{\mathfrak{s}}_k): k\geq 1\}$, we have the inclusion $C_c^\infty(\mathbb{R})\subset \tilde{\FF}$.
 For any function $\varphi \in C_c^\infty(\mathbb{R})$, we have
 \[
 \begin{aligned}
 	\tilde{\EE}(\varphi, \varphi)&= \frac{1}{2}\sum_{k} \int_{\bar{\mathtt{I}}_k} \varphi'(x)^2\frac{1}{\bar{\mathfrak{s}}'_k(x)}dx  \\
 	&=\frac{1}{2}\sum_{n} \int_{I_n}\varphi'(x)^2 \left(\frac{dx}{d\ss_n} \right)^2d\ss_n \\
 	&=\frac{1}{2}\sum_{n} \int_{I_n} \left(\frac{d\varphi}{d\ss_n} \right)^2d\ss_n \\
 	&=\EE(\varphi, \varphi)=\bar{\EE}(\varphi, \varphi).
 \end{aligned}\]
 Finally, we need only to prove $C_c^\infty(\mathbb{R})$ is a special standard core of $(\tilde{\EE},\tilde{\FF})$. By the definition of the equivalence class, we can easily deduce that each $\bar{\mathtt{I}}_k$ is scale-isolated with respect to $\{\bar{\mathfrak{s}}_k: k\geq 1\}$. Notice that $\bar{\mathfrak{s}}_k$ is absolutely continuous. By applying Theorem~\ref{THM37} to $\{(\bar{\mathtt{I}}_k, \bar{\mathfrak{s}}_k): k\geq 1\}$, we conclude that $C_c^\infty(\mathbb{R})$ is a special standard core of $(\tilde{\EE},\tilde{\FF})$. That completes the proof.
 \end{proof}

 \begin{remark}
 Consider the following example. Let $I_1:=(0,1]$ and $\ss_1$ be the scale function on $I_1$ given by \eqref{EQ3SNX2}. Further let $(\EE,\FF)$ be the strongly local and regular Dirichlet form on $L^2(\mathbb{R})$ with only one effective interval $(I_1, \ss_1)$, i.e.
 \[
 \begin{aligned}
 	\FF&:=\left\{u\in L^2(\mathbb{R}): u|_{(0,1]} \in \FF^{(\ss_1)} \right\},\\
 	\EE(u,v)&:= \frac{1}{2}\int_0^1 \frac{du|_{(0,1]}}{d\ss_1}\frac{dv|_{(0,1]}}{d\ss_1}d\ss_1,\quad u,v \in \FF,
\end{aligned} \]
where $\FF^{(\ss_1)}$ is given by \eqref{EQ2FSU} with $J=I_1, \ss=\ss_1$. Clearly, $C_c^\infty(\mathbb{R})\subset \FF$. Note that the absolutely continuous part of $\ss_1$ is $\bar{\ss}_1(x)=x$. Thus the $C^\infty$-merging of the unique equivalence class is $(\bar{\mathtt{I}}_1, \bar{\mathfrak{s}}_1)$ with $\bar{\mathtt{I}}_1=[0,1]$ and $\bar{\mathfrak{s}}_1(x)=x$. It follows that the closure of $C_c^\infty(\mathbb{R})$ in $(\EE,\FF)$ is
\[
\begin{aligned}
	&\bar{\FF}=\left\{u\in L^2(\mathbb{R}): u|_{[0,1]}\in H^1([0,1])\right\},\\
	&\bar{\EE}(u,v)=\frac{1}{2}\int_0^1 u'(x)v'(x)dx,\quad u,v\in \bar{\FF}.
\end{aligned}\]
The associated diffusion process on the effective interval $[0,1]$ is the reflected Brownian motion.

In this example, $0$ is a trap relative to $(\EE,\FF)$ and thus $\{0\}$ is an exceptional set relative to $(\EE,\FF)$. However, $\{0\}$ is not an exceptional set relative to $(\bar{\EE},\bar{\FF})$. Notice that $(\bar{\EE},\bar{\FF})$ is a regular Dirichlet subspace of $(\EE,\FF)$. As stated in \cite[\S2]{LY16}, if $(\EE^1,\FF^1)$ is a regular Dirichlet subspace of $(\EE^2,\FF^2)$ and denote their $1$-capacities by $\text{Cap}^1$ and $\text{Cap}^2$ respectively, then $\text{Cap}^1(A)\geq \text{Cap}^2(A)$ for any appropriate set $A$. The above example tells us the strict inequality may hold. Intuitively speaking, a part of the state space with no sense for a Dirichlet form could be very important to its regular Dirichlet subspace (of course, not vice versa).
 \end{remark}

 \section{Killing inside}\label{SEC5}

 All the discussions above concern only the strongly local Dirichlet forms. That means the associated diffusion processes have no killing inside. This section is devoted to explain the following fact: the presence of killing inside has no influence for all the results above. The main tools we shall use are the resurrected transform and killing transform of the Dirichlet form, see \cite[Theorems~5.2.17 and 5.1.6]{CF12}. Note that any function in this section is taken as its quasi-continuous version.

We first consider a regular and local Dirichlet form $(\EE,\FF)$ on $L^2(I,m)$, where $I=\langle l, r \rangle$ and $m$ is a fully supported Radon measure on $I$. Let $k$ be the killing measure of $(\EE,\FF)$ in its Beurling-Deny decomposition. Then $k$ is a Radon smooth measure on $I$ relative to $(\EE,\FF)$. The resurrected Dirichlet form of $(\EE,\FF)$ can be deduced as follows. Let $\FF_\mathrm{e}$ be the extended Dirichlet space of $(\EE,\FF)$ and define
\[
	\EE^\mathrm{res}(u,v):=\EE(u,v)-\int_I u(x)v(x)k(dx),\quad u,v\in \FF_\mathrm{e}.
\]
 Clearly, $(\EE^\mathrm{res}, \FF_\mathrm{e}\cap L^2(I,m+k))$ is a regular Dirichlet form on $L^2(I,m+k)$. Denote the extended Dirichlet space of this Dirichlet form by $\FF^\mathrm{res}_\mathrm{e}$ and set
 \[
 \FF^\mathrm{res}:= \FF^\mathrm{res}_\mathrm{e}\cap L^2(I,m).
 \]
 Then $(\EE^\mathrm{res}, \FF^\mathrm{res})$ is a regular and strongly local Dirichlet form on $L^2(I,m)$ sharing the same set of quasi notions with $(\EE,\FF)$ (Cf. \cite[Theorem~5.2.17]{CF12}). It is called the resurrected Dirichlet form of $(\EE,\FF)$. Furthermore, we have
 \begin{equation}\label{EQ4FLI}
 	\FF^\mathrm{res}\cap L^2(I,k) =\FF^\mathrm{res}_\mathrm{e}\cap L^2(I, m+k) =\FF_\mathrm{e}\cap L^2(I,m+k)=\FF
 \end{equation}
and
\begin{equation}\label{EQ4EUV}
	\EE(u,v)=\EE^\mathrm{res}(u,v)+\int_I u(x)v(x)k(dx),\quad u,v\in \FF^\mathrm{res}\cap L^2(I,k)=\FF.
\end{equation}
Thus on the contrary, $(\EE,\FF)$ is the killed Dirichlet form of $(\EE^\mathrm{res},\FF^\mathrm{res})$ induced by $k$. Since Theorem~\ref{THM21} gives an expression of $(\EE^\mathrm{res}, \FF^\mathrm{res})$, we can obtain the following characterization of $(\EE,\FF)$.

\begin{theorem}\label{THM51}
Let $I=\langle l, r\rangle$ be an interval and $m$ a fully supported Radon measure on $I$. Then $(\EE,\FF)$ is a regular and local Dirichlet form on $L^2(I, m)$ if and only if there exist a set of at most countable disjoint intervals $\{I_n=\langle a_n,b_n\rangle: I_n\subset I, n\geq 1\}$, a scale function $\ss_n\in \SS_\infty(I_n)$ for each $n\geq 1$ and  a Radon measure $k$ on $I$ with $k\ll m$ on $I\setminus \cup_{n \geq 1}I_n$, such that
\begin{equation}\label{EQ4FUL}
\begin{aligned}
	&\FF=\left\{u\in L^2(I,m+k): u|_{I_n}\in \FF^{(\ss_n)}, \sum_{n\geq 1}\EE^{(\ss_n)}(u|_{I_n}, u|_{I_n})<\infty  \right\},  \\
	&\EE(u,v)=\sum_{n\geq 1}\EE^{(\ss_n)}(u|_{I_n}, v|_{I_n})+\int_I u(x)v(x)k(dx),\quad u,v \in \FF,
\end{aligned}
\end{equation}
where for each $n\geq 1$, $(\EE^{(\ss_n)}, \FF^{(\ss_n)})$ is given by \eqref{EQ2FSU} with the scale function $\ss_n$.  Moreover, the intervals $\{I_n: n\geq 1\}$ and scale functions $\{\ss_n:n\geq 1\}$ are uniquely determined, if the difference of order is ignored.
\end{theorem}
\begin{proof}
 For the necessity, let $k$ be the killing measure of $(\EE,\FF)$ and consider the resurrected Dirichlet form $(\EE^\mathrm{res},\FF^\mathrm{res})$ of $(\EE,\FF)$.  Since $(\EE^\mathrm{res},\FF^\mathrm{res})$ is a regular and strongly local Dirichlet form on $L^2(I,m)$, it follows that $(\EE^\mathrm{res},\FF^\mathrm{res})$ is expressed as \eqref{EQ2FUL} with a class of effective intervals $\{(I_n,\ss_n): n\geq 1\}$. Since $(\EE,\FF)$ and $(\EE^\mathrm{res}, \FF^\mathrm{res})$ share the same set of quasi notions, and  the subset of $I\setminus \cup_{n\geq 1}I_n$ with zero $m$-measure is exceptional relative to $(\EE^\mathrm{res}, \FF^\mathrm{res})$, we can conclude $k\ll m$ on $I\setminus \cup_{n\geq 1}I_n$.
 Finally, \eqref{EQ4FLI} and \eqref{EQ4EUV} implies the expression \eqref{EQ4FUL} of $(\EE,\FF)$.

For the sufficiency, we consider the regular Dirichlet form \eqref{EQ2FUL} on $L^2(I,m)$. Note that the exceptional subset of $I_n$ relative to \eqref{EQ2FUL} must be the empty set. Then $k\ll m$ on $I\setminus \cup_{n \geq 1}I_n$ implies $k$ is a Radon smooth measure relative to \eqref{EQ2FUL} and thus \eqref{EQ4FUL} is the killed Dirichlet form of \eqref{EQ2FUL} induced by $k$. Therefore, \eqref{EQ4FUL} is a regular and local Dirichlet form on $L^2(I, m)$. That completes the proof.
\end{proof}

We end this section with some explanations for the extensions of Theorems~\ref{THM37} and \ref{THM414} to a local Dirichlet form. As proved above, \eqref{EQ2FUL} is the resurrected Dirichlet form of \eqref{EQ4FUL} and \eqref{EQ4FUL}  is the killed Dirichlet form of \eqref{EQ2FUL} induced by $k$. From \cite[Theorems~5.1.6 and 5.2.17]{CF12}, we know that a class $\mathcal{C}$ is a special standard core of \eqref{EQ2FUL} if and only if $\mathcal{C}$ is a special standard core of \eqref{EQ4FUL}. Therefore, nothing else need to be changed if the Dirichlet forms in Theorems~\ref{THM37} and \ref{THM414} are only assumed to enjoy the local property.

\section*{Acknowledgement}

Example~\ref{EXA21} was inspired by the discussions with Professor Zhi-Ming Ma, and the first named author would like to thank his helpful advice.

\end{document}